\documentclass[12pt,reqno]{amsart}

\usepackage[cp1251]{inputenc}
\usepackage{amsthm,amssymb,amsmath,amsfonts}
\usepackage{graphicx}
\usepackage[matrix,arrow,curve]{xy}
\newcommand{\Q}{\ensuremath{\mathbb{Q}}}

\newcommand{\Z}{\ensuremath{\mathbb{Z}}}

\newcommand{\A}{\ensuremath{\mathbb{A}}}

\newcommand{\CG}{\ensuremath{\mathfrak{C}}}

\newcommand{\DG}{\ensuremath{\mathfrak{D}}}

\newcommand{\AG}{\ensuremath{\mathfrak{A}}}

\newcommand{\SG}{\ensuremath{\mathfrak{S}}}

\newcommand{\VG}{\ensuremath{\mathfrak{V}}}

\newcommand{\ka}{\ensuremath{\Bbbk}}

\newcommand{\kka}{\ensuremath{\overline{\Bbbk}}}

\newcommand{\XX}{\ensuremath{\overline{X}}}

\newcommand{\Pro}{\ensuremath{\mathbb{P}}}

\newcommand{\Aut}{\ensuremath{\operatorname{Aut}}}

\newcommand{\Gal}{\ensuremath{\operatorname{Gal}}}

\newcommand{\Pic}{\ensuremath{\operatorname{Pic}}}

\newcommand{\ord}{\ensuremath{\operatorname{ord}}}

\makeatletter
\@addtoreset{equation}{section}
\makeatother

\newtheorem{theorem}[equation]{Theorem}
\newtheorem{proposition}[equation]{Proposition}
\newtheorem{lemma}[equation]{Lemma}
\newtheorem{corollary}[equation]{Corollary}

\theoremstyle{definition}
\newtheorem{example}[equation]{Example}
\newtheorem{definition}[equation]{Definition}

\theoremstyle{remark}
\newtheorem{remark}[equation]{Remark}
\newtheorem*{notation}{Notation}

\title{Quotients of cubic surfaces}
\address{Institute for Information Transmission Problems, 19 Bolshoy Karetnyi side-str., Moscow 127994, Russia}
\address{Laboratory of Algebraic Geometry, National Research University Higher School of Economics, 7 Vavilova str., Moscow 117312, Russia}
\email{trepalin@mccme.ru}
\thanks{The article was prepared within the framework of a subsidy granted to the HSE by the Government of the Russian Federation for the implementation of the Global Competitiveness Program, and the grants RFFI 15-01-02164-a and N.SH.-2998.2014.1}
\author{Andrey Trepalin}

\begin{document}

\begin{abstract}
Let $\ka$ be any field of characteristic zero, $X$ be a cubic surface in $\Pro^3_{\ka}$ and $G$ be a group acting on $X$. We show that if $X(\ka) \ne \varnothing$ and $G$ is not trivial and not a group of order $3$ acting in a special way then the quotient surface $X / G$ is rational over~$\ka$. For the group $G$ of order $3$ we construct examples of both rational and nonrational quotients of both rational and nonrational $G$-minimal cubic surfaces over $\ka$.
\end{abstract}

\maketitle
\section{Introduction}

Let $\ka$ be any field of characteristic zero. If $\ka$ is algebraically closed then any quotient of a rational surface by an action of a finite group is rational by Castelnuovo criterion. For del Pezzo surfaces of degree~$4$ and higher the following theorem holds.

\begin{theorem}[{\cite[Theorem 1.1]{Tr13}}]
\label{ratquot4}
Let $\ka$ be a field of characteristic zero, $X$ be a del Pezzo surface over $\ka$ such that $X(\ka) \ne \varnothing$ and $G$ be a finite subgroup of automorphisms of $X$. If $K_X^2 \geqslant 5$ then the quotient variety $X / G$ is $\ka$-rational. If $K_X^2 = 4$ and the order of $G$ is not equal to $1$, $2$ or $4$ then $X / G$ is $\ka$-rational.
\end{theorem}

In this paper we find for which finite groups a quotient of cubic surface is $\ka$-rational and for which is not. The main result of this paper is the following.

\begin{theorem}
\label{DP3}
Let $\ka$ be a field of characteristic zero, $X$ be a del Pezzo surface over $\ka$ of degree $3$ such that $X(\ka) \ne \varnothing$ and $G$ be a subgroup of $\Aut_{\ka}(X)$. Suppose that $G$ is not trivial and $G$ is not a group of order $3$ having no curves of fixed points. Then $X / G$ is $\ka$-rational.
\end{theorem}

Note that if $G$ is trivial and $X$ is minimal then $X$ is not $\ka$-rational (see \cite[Theorem~V.1.1]{Man74}). This gives us an example of a del Pezzo surface of degree $3$ such that its quotient by the trivial group is not $\ka$-rational. For a group $G$ of order $3$ acting without curves of fixed points on $X$ we construct examples of quotients of $G$-minimal cubic surface~$X$ such that $X$ is $\ka$-rational and $X / G$ is $\ka$-rational, $X$ is $\ka$-rational and~$X / G$ is not \mbox{$\ka$-rational}, $X$~is not $\ka$-rational and~$X / G$ is $\ka$-rational, and $X$ is not $\ka$-rational and $X / G$ is not $\ka$-rational.

To prove Theorem \ref{DP3} we consider possibilities for groups $G$ acting on $X$. Our main method is to find a normal subgroup $N$ in $G$ such that the quotient $X / N$ is \mbox{$G / N$-birationally} equivalent to a del Pezzo surface of degree $5$ or more. Therefore \mbox{$\ka$-rationality} of~$X / G$ is equivalent to $\ka$-rationality of the quotient of the obtained del Pezzo surface by the group $G / N$ and we can use Theorem \ref{ratquot4}.

The plan of this paper is as follows. In Section $2$ we recall some facts about minimal rational surfaces, groups, singularities and quotiens. In Section $3$ we consider quotients of cubic surfaces by nontrivial groups of automorphisms and show that all them except a case are always $\ka$-rational. In Section $4$ for non-$\ka$-rational quotients of a $\ka$-rational cubic surface $X$ by a group of order $3$ we find all possibilities of the image of the Galois group $\Gal\left(\kka / \ka\right)$ in the Weyl group $W(E_6)$ acting on the Picard group of $X$. In Section $5$ we find an explicit geometric interpretation of the obtained actions of the Galois group in terms of equations of $X$. In Section $6$ for a group $G$ of order $3$ acting on a $G$-minimal cubic surface $X$ without curves of fixed points we construct examples of $\ka$-rational and non-$\ka$-rational quotients.

The author is grateful to his adviser Yu.\,G.\,Prokhorov and to C.\,A.\,Shramov for useful discussions.

\begin{notation}

Throughout this paper $\ka$ is any field of characteristic zero, $\kka$ is its algebraic closure. For a surface $X$ we denote $X \otimes \kka$ by $\XX$. For a surface $X$ we denote the Picard group (resp. $G$-invariant Picard group) by $\Pic(X)$ (resp. $\Pic(X)^G$). The number \mbox{$\rho(X) = \operatorname{rk} \Pic(X)$} (resp. \mbox{$\rho(X)^G = \operatorname{rk} \Pic(X)^G$}) is the Picard number (resp. $G$-invariant Picard number) of $X$. If two surfaces $X$ and $Y$ are $\ka$-birationally equivalent then we write~$X \approx Y$. If two divisors $A$ and $B$ are linearly equivalent then we write $A \sim B$.

\end{notation}

\section{Preliminaries}

\subsection{$G$-minimal rational surfaces}


In this subsection we review main notions and results of $G$-equivariant minimal model program following the papers \cite{Man67}, \cite{Isk79}, \cite{DI1}. Throughout this subsection $G$ is a finite group.

\begin{definition}
\label{rationality}
A {\it rational variety} $X$ is a variety over $\ka$ such that $\XX=X \otimes \kka$ is birationally equivalent to $\Pro^n_{\kka}$.

A {\it $\ka$-rational variety} $X$ is a variety over $\ka$ such that $X$ is birationally equivalent to $\Pro^n_{\ka}$.

A variety $X$ over $\ka$ is a {\it $\ka$-unirational variety} if there exists a $\ka$-rational variety $Y$ and a dominant rational map $\varphi: Y \dashrightarrow X$.
\end{definition}

\begin{definition}
\label{minimality}
A {\it $G$-surface} is a pair $(X, G)$ where $X$ is a projective surface over $\ka$ and~$G$ is a finite subgroup of $\Aut_{\ka}(X)$. A morphism of $G$-surfaces $f: X \rightarrow X'$ is called a \textit{$G$-morphism} if for each $g \in G$ one has $fg = gf$.

A smooth $G$-surface $(X, G)$ is called {\it $G$-minimal} if any birational morphism of smooth $G$-surfaces $(X, G) \rightarrow (X',G)$ is an isomorphism.

Let $(X, G)$ be a smooth $G$-surface. A $G$-minimal surface $(Y, G)$ is called a {\it minimal model} of $(X, G)$ or {\it $G$-minimal model} of $X$ if there exists a birational \mbox{$G$-morphism $X \rightarrow Y$}.
\end{definition}

The following theorem is a classical result about $G$-equivariant minimal model program.

\begin{theorem}
\label{GMMP}
Any $G$-morphism $f:X \rightarrow Y$ can be factorized in the following way:
$$
X= X_0 \xrightarrow{f_0} X_1 \xrightarrow{f_1} \ldots \xrightarrow{f_{n-2}} X_{n-1} \xrightarrow{f_{n-1}} X_n = Y,
$$
where each $f_i$ is a contraction of a set $\Sigma_i$ of disjoint $(-1)$-curves on~$X_i$, such that $\Sigma_i$ is defined over $\ka$ and $G$-invariant. In particular,
$$
K_Y^2 - K_X^2 \geqslant \rho(X)^G - \rho(Y)^G.
$$
\end{theorem}

The classification of $G$-minimal rational surfaces is well-known due to V.\,Iskovskikh and Yu.\,Manin (see \cite{Isk79} and \cite{Man67}). We introduce some important notions before surveying it.

\begin{definition}
\label{Cbundledef}
A smooth rational $G$-surface $(X, G)$ admits a {\it conic bundle} structure if there exists a $G$-morphism $\varphi: X \rightarrow B$ such that any scheme fibre is isomorphic to a reduced conic in~$\Pro^2_{\ka}$ and $B$ is a smooth curve.
\end{definition}

\begin{definition}
\label{DPdef}
A {\it del Pezzo surface} is a smooth projective surface~$X$ such that the anticanonical divisor $-K_X$ is ample. A {\it singular del Pezzo surface} is a normal projective surface $X$ such that the anticanonical divisor $-K_X$ is ample and all singularities of $X$ are Du Val singularities. A {\it weak del Pezzo surface} is a smooth projective surface $X$ such that the anticanonical divisor $-K_X$ is nef and big.

The number $d = K_X^2$ is called the {\it degree} of a (singular) del Pezzo surface $X$.
\end{definition}

A del Pezzo surface $X$ of degree $3$ is isomorphic to a smooth cubic surface in $\Pro^3_{\ka}$.

The following theorem classifies $G$-minimal rational surfaces.

\begin{theorem}[{\cite[Theorem 1]{Isk79}}]
\label{Minclass}
Let $X$ be a $G$-minimal rational $G$-surface. Then either $X$ admits a $G$-equivariant conic bundle structure with $\Pic(X)^{G} \cong \Z^2$, or $X$ is a del Pezzo surface with $\Pic(X)^{G} \cong \Z$.
\end{theorem}

\begin{theorem}[{cf. \cite[Theorem 4]{Isk79}}]
\label{MinCB}
Let $X$ admit a $G$-equivariant conic bundle structure. Suppose that $K_X^2 = 3, 5, 6$ or $X$ is a blowup of $\Pro^2_{\ka}$ at a point. Then $X$ is not $G$-minimal.
\end{theorem}

The following theorem is an important criterion of $\ka$-rationality over an arbitrary perfect field $\ka$.

\begin{theorem}[{\cite[Chapter 4]{Isk96}}]
\label{ratcrit}
A minimal rational surface $X$ over a perfect field $\ka$ is $\ka$-rational if and only if the following two conditions are satisfied:

(i) $X(\ka) \ne \varnothing$;

(ii) $K_X^2 \geqslant 5$.
\end{theorem}

\begin{corollary}
\label{piccrit}
Let $X$ be a rational $G$-surface such that $X(\ka) \ne \varnothing$ and $\rho(X)^G + K_X^2 \geqslant 7$. Then there exists a $G$-minimal model $Y$ of $X$ such that $K_Y^2 \geqslant 6$. In particular, $X$ is $\ka$-rational.
\end{corollary}
\begin{proof}
By Theorem \ref{Minclass} there exists a birational $G$-morphism \mbox{$f: X \rightarrow Z$} such \mbox{that $\rho(Z)^G \leqslant 2$}. By Theorem \ref{GMMP} one has
$$
K_Z^2 \geqslant K_X^2 + \rho(X)^G - \rho(Z)^G \geqslant 7 - \rho(Z)^G.
$$

If $\rho(Z)^G = 1$ then $K_Z^2 \geqslant 6$. \mbox{If $\rho(Z)^G = 2$} and $K_Z^2 = 5$ then $Z$ is not $G$-minimal by Theorem \ref{MinCB}. Therefore there exists a $G$-minimal model $Y$ of $Z$ such that $K_Y^2 \geqslant 6$.

The set $X(\ka)$ is not empty. Thus $Y(\ka) \ne \varnothing$ and $X \approx Y$ is $\ka$-rational by Theorem~\ref{ratcrit}.
\end{proof}

In this paper we use the notation of the following remark.

\begin{remark}
\label{DP_1curves}
Let $X$ be a cubic surface in $\Pro^3_{\ka}$. Then $\XX$ can be realized as a blowup \mbox{$f: \XX \rightarrow \Pro^2_{\kka}$} at $6$ points $p_1$, $\ldots$, $p_6$ in general position. Put $E_i = f^{-1}(p_i)$ and $L = f^*(l)$, where~$l$ is the class of a line on $\Pro^2_{\kka}$. One has
$$
-K_{\XX} \sim 3L - \sum \limits_{i=1}^6 E_i.
$$
The $(-1)$-curves on $\XX$ are $E_i$, the proper transforms \mbox{$L_{ij} \sim L - E_i - E_j$} of the lines passing through a pair of points $p_i$ and $p_j$, and the proper transforms
$$
Q_j \sim 2L + E_j - \sum \limits_{i = 1}^6 E_i
$$
\noindent of the conics passing through five points from the set $\{ p_1, p_2, p_3, p_4, p_5, p_6\}$.

In this notation one has:
$$
E_i \cdot E_j = 0; \qquad E_i \cdot L_{ij} = 1; \qquad E_i \cdot L_{jk} = 0;
$$
$$
L_{ij} \cdot L_{ik} = 0; \qquad L_{ij} \cdot L_{kl} = 1; \qquad E_i \cdot Q_i = 0; \qquad E_i \cdot Q_j = 1;
$$
$$
Q_i \cdot Q_j = 0; \qquad Q_i \cdot L_{ij} = 1; \qquad Q_i \cdot L_{jk} = 0.
$$
\noindent where $i$, $j$ and $k$ are different numbers from the set $\{1, 2, 3, 4, 5, 6\}$.
\end{remark}

\subsection{Groups}

In this subsection we collect some results and notation concerning groups used in this paper.

We use the following notation:

\begin{itemize}

\item $\CG_n$ denotes the cyclic group of order $n$;

\item $\DG_{2n}$ denotes the dihedral group of order $2n$;

\item $\SG_n$ denotes the symmetric group of degree $n$;

\item $\AG_n$ denotes the alternating group of degree $n$;

\item $(i_1 i_2 \ldots i_j)$ denotes a cyclic permutation of $i_1$, \ldots, $i_j$;

\item $\VG_4$ denotes the Klein group isomorphic to $\CG_2^2$;

\item $\langle g_1, \ldots, g_n \rangle$ denotes a group generated by $g_1$, \ldots, $g_n$;

\item $\operatorname{diag}(a_1, \ldots, a_n)$ denotes the diagonal $n \times n$ matrix with entries $a_1$, \ldots $a_n$;

\item $\mathrm{i} = \sqrt{-1}$;

\item $\xi_n = e^{\frac{2\pi \mathrm{i}}{n}}$

\item $\omega = \xi_3 = e^{\frac{2\pi \mathrm{i}}{3}}$.

\end{itemize}

The group $\SG_5$ can act on a cubic surface. Therefore it is important to know some facts about subgroups of this group. The following lemma is an easy exercise.

\begin{lemma}
\label{S5subgroups}
Any nontrivial subgroup $G \subset \SG_5$ contains a normal subgroup $N$ conjugate in $\SG_5$ to one of the following groups:

\begin{itemize}

\item $\CG_2 \cong \langle(12)\rangle$,
\item $\CG_2 \cong \langle(12)(34)\rangle$,
\item $\CG_3 \cong \langle(123)\rangle$,
\item $\VG_4 \cong \langle(12)(34), (13)(24)\rangle$,
\item $\CG_5 \cong \langle(12345)\rangle$,
\item $\AG_5$.

\end{itemize}

\end{lemma}

\subsection{Singularities}

In this subsection we review some results about quotient singularities and their resolutions.

All singularities appearing in this paper are toric singularities. These singularities are locally isomorphic to the quotient of $\A^2$ by a cyclic group generated by $\operatorname{diag}(\xi_m, \xi_m^q)$. Such a singularity is denoted by $\frac{1}{m}(1,q)$. If $\gcd(m,q) > 1$ then the group
$$
\CG_m \cong \langle \operatorname{diag}(\xi_m, \xi_m^q) \rangle
$$
\noindent contains a reflection and the quotient singularity is isomorphic to a quotient singularity with smaller $m$.

A toric singularity can be resolved by a sequence of weighted blowups. Therefore it is easy to describe numerical properties of a quotient singularity. We list here these properties for singularities appearing in our paper.

\begin{remark}
Let the group $\CG_m$ act on a smooth surface $X$ and \mbox{$f: X \rightarrow S$} be a quotient map. Let $p$ be a singular point on $S$ of type $\frac{1}{m}(1,q)$. Let $C$ and $D$ be curves passing through $p$ such that $f^{-1}(C)$ and $f^{-1}(D)$ are $\CG_m$-invariant and tangent vectors of these curves at the point $f^{-1}(p)$ are eigenvectors of the natural action of $\CG_m$ on $T_{f^{-1}(p)} X$ (the curve $C$~corresponds to the eigenvalue $\xi_m$ and the curve $D$ corresponds to the eigenvalue~$\xi_m^q$).

Let $\pi: \widetilde{S} \rightarrow S$ be the minimal resolution of the singular point $p$. Table \ref{table1} presents some numerical properties of $\widetilde{S}$ and $S$ for the singularities appearing in this paper.

The exceptional divisor of $\pi$ is a chain of transversally intersecting exceptional curves~$E_i$ whose selfintersection numbers are listed in the last column of Table \ref{table1}. The curves~$\pi^{-1}_*(C)$ and $\pi^{-1}_*(D)$ transversally intersect at a point only the first and the last of these curves respectively and do not intersect other components of the exceptional divisor of~$\pi$.

\begin{table}
\caption{} \label{table1}

\begin{tabular}{|c|c|c|c|c|c|}
\hline
$m$ & $q$ & $K_{\widetilde{S}}^2 - K_S^2$ & $\pi^{-1}_*(C)^2 - C^2$ & $\pi^{-1}_*(D)^2 - D^2$ \rule[-7pt]{0pt}{20pt} & $E_i^2$ \\
\hline
$2$ & $1$ & $0$ & $-\dfrac{1}{2}$ & $-\dfrac{1}{2}$ \rule[-11pt]{0pt}{30pt} & $-2$ \\
\hline
$3$ & $1$ & $-\dfrac{1}{3}$ & $-\dfrac{1}{3}$ & $-\dfrac{1}{3}$ \rule[-11pt]{0pt}{30pt} & $-3$ \\
\hline
$3$ & $2$ & $0$ & $-\dfrac{2}{3}$ & $-\dfrac{2}{3}$ \rule[-11pt]{0pt}{30pt} & $-2$, $-2$ \\
\hline
$5$ & $2$ & $-\dfrac{2}{5}$ & $-\dfrac{2}{5}$ & $-\dfrac{3}{5}$ \rule[-11pt]{0pt}{30pt} & $-3$, $-2$ \\
\hline
\end{tabular}

\end{table}

\end{remark}

\subsection{Quotients}

In this subsection we collect some additional information about quotients of rational surfaces.

We use the following definition for convenience.

\begin{definition}
\label{MMPred}
Let $X$ be a $G$-surface (resp. surface), $\widetilde{X} \rightarrow X$ be its minimal resolution of singularities, and $Y$ be a $G$-minimal model (resp. minimal model) of $\widetilde{X}$. We call the surface $Y$ a \textit{$G$-MMP-reduction} (resp. \textit{MMP-reduction}) of $X$.
\end{definition}

We need some results about quotients of del Pezzo surfaces of degree $4$.

\begin{lemma}[{\cite[Remark 6.2]{Tr13}}]
\label{DP4i12}
Let a finite group $G$ act on a del Pezzo surface $X$ of degree $4$ and $N \cong \CG_2$ be a normal subgroup in $G$ such that $N$ has no curves of fixed points. Then the surface $X / N$ is $G / N$-birationally equivalent to a conic bundle $Y$ with $K_Y^2 = 2$. If there exists a $G \times \Gal \left( \kka / \ka \right)$-fixed point then $Y$ is not $G / N$-minimal and there exists a $G / N$-MMP-reduction $Z$ of $Y$ such that $K_Z^2 = 8$.
\end{lemma}

\section{Del Pezzo surface of degree $3$}

In this Section we prove Theorem \ref{DP3}. We start from cyclic groups of prime order. The following theorem classifies actions of cyclic groups of prime order on smooth cubics.

\begin{theorem}[{cf. \cite[Theorem~6.10]{DI1}}]
\label{DP3cyclicclass}
Let a group $\CG_p$ of prime order $p$ act on a del Pezzo surface of degree $3$. Then one can choose homogeneous coordinates $x$, $y$, $z$, $t$ in $\Pro^3_{\kka}$ such that the equation of $\XX$ and the action of $\CG_p$ are presented in Table \ref{table2}, where $u$, $v$, $w$, $\alpha$ and $\beta$ are coefficients. These actions have different sets of fixed points on $\XX$ and correspond to different conjugacy classes of cyclic subgroups in the Weyl group $W(E_6)$ acting on $\Pic(\XX)$.

\begin{table}
\caption{} \label{table2}

\begin{center}
\begin{tabular}{|c|c|c|c|}
\hline
Type & Order & Equation & Action \\
\hline
$1$ & $2$ & $x^3 + y^3 + z^3 + \alpha xyz + t^2(ux + vy + wz) = 0$ & $\left( x : y : z : -t \right)$ \\
\hline
$2$ & $2$ & $x^3 + y^3 + xz(z + \alpha t) + yt(z + \beta t) = 0$ & $\left( x : y : -z : -t \right)$ \\
\hline
$3$ & $3$ & $x^3 + y^3 + z^3 + \alpha xyz + t^3 = 0$ & $\left( x : y : z : \omega t \right)$ \\
\hline
$4$ & $3$ & $x^3 + y^3 + z^3 + t^3 = 0$ & $\left( x : y : \omega z : \omega t \right)$ \\
\hline
$5$ & $3$ & $x^3 + y^3 + zt(ux + vy) + z^3 + t^3 = 0$ & $\left( x : y : \omega z : \omega^2 t \right)$ \\
\hline
$6$ & $5$ & $x^2 y + y^2 z + z^2 t + t^2 x = 0$ & $\left( x :\xi_5 y : \xi_5^4 z : \xi_5^3 t \right)$ \\
\hline

\end{tabular}
\end{center}

\end{table}

\end{theorem}

In this section we prove Theorem \ref{DP3}. Note that elements of type $3$ and $4$ of Table \ref{table2} have curves of fixed points $t = 0$ and $x = y = 0$ respectively. Therefore an element of order $3$ having no curves of fixed points has type $5$ of Table \ref{table2}.

In the latter case the following lemma holds.

\begin{lemma}
\label{DP3C35}
Let a finite group $G$ act on a del Pezzo surface $X$ of degree $3$ and $N \cong \CG_3$ be a normal subgroup in $G$ such that $N$ acts as in type $5$ of Table \ref{table2}. Then the surface $X / N$ is $G / N$-birationally equivalent to a del Pezzo surface of degree $3$.
\end{lemma}

\begin{proof}
Let $\XX$ be given by equation
$$
x^3 + y^3 + zt(ux + vy) + z^3 + t^3 = 0
$$
\noindent in $\Pro^3_{\kka}$ and $N$ act as
$$
(x : y : z: t) \mapsto \left( x : y : \omega z : \omega^2 t \right).
$$
\noindent The fixed points of $N$ lie on the line $z = t = 0$. Thus $N$ has three fixed points $q_1$, $q_2$ and $q_3$. One can easily check that on the tangent spaces of $\XX$ at these points $N$ acts as~$\langle\operatorname{diag}(\omega, \omega^2)\rangle$. Denote by $C_1$ and $C_2$ invariant curves $z = 0$ and $t = 0$ each passing through the three points $q_i$.

Let $f: X \rightarrow X / N$ be the quotient morhism and
$$
\pi: \widetilde{X / N} \rightarrow X / N
$$
\noindent be the minimal resolution of singularities. The curves $f(C_1)$ and $f(C_2)$ meet each other at the three singular points of $X / N$ and $f(C_1) \cdot f(C_2) = 1$. Thus two curves $\pi^{-1}_* f (C_j)$ are disjoint. Moreover (see Table \ref{table1}), one has
$$
\pi^{-1}_* f (C_j)^2 = f(C_j)^2 - 3 \cdot \frac{2}{3} = \frac{1}{3} C_j^2 - 2 = -1.
$$
Therefore we can $G / N$-equivariantly contract the two $(-1)$-curves $\pi^{-1}_* f (C_j)$ and get a surface~$Y$ with $K_Y^2 = 3$.

The surface $X / N$ has only Du Val singularities. Therefore $X / N$ is a singular del Pezzo surface and $\widetilde{X / N}$ is a weak del Pezzo surface containing exactly six curves $\pi^{-1}(q_i)$ whose selfintersection is less than $-1$. Thus $Y$ does not contain curves with selfintersection less than $-1$. So $Y$ is a del Pezzo surface of degree $3$.
\end{proof}

\begin{remark}
\label{Eckardtpoints}
Note that in the notation of Lemma \ref{DP3C35} there are two points on the surface $Y$ where three $(-1)$-curves meet each other. These points are images of $\pi^{-1}_* f (C_j)$. Such a point is called an \textit{Eckardt point} (see Definition \ref{Eckardtpt} below).
\end{remark}

\begin{remark}
\label{DP3C35rat}
Note that in the notation of Lemma \ref{DP3C35} if $\rho(X)^G > 1$ then $X$ is not \mbox{$G$-minimal} by Theorem \ref{MinCB}. Therefore the quotient of $X / N$ is equivalent to a quotient of a del Pezzo surface with degree greater than $3$ by a group of order $3$. By Theorem~\ref{ratquot4} such a quotient is $\ka$-rational.
\end{remark}

In Section $4$ for non-$\ka$-rational quotient $X / \CG_3$ of $\ka$-rational surface $X$ we find restrictions on the image of the Galois group $\Gal\left(\kka / \ka\right)$ in the Weyl group $W(E_6)$ which acts on~$\Pic(\XX)$.

Now we show that in all other cases of Theorem \ref{DP3} the quotient of $X$ is $\ka$-rational.

\begin{lemma}
\label{DP3cyclic}
Let a finite group $G$ act on a del Pezzo surface $X$ of degree $3$ and $N \cong \CG_p$ be a normal cyclic subgroup of prime order in $G$ such that $N$ acts not as in type $5$ of Table~\ref{table2}. Then there exists a \mbox{$G / N$-MMP-reduction $Y$} of $X / N$ such that $K_Y^2 \geqslant 5$.
\end{lemma}

\begin{proof}
Let us consider the possibilities case by case.

In types $1$ and $3$ of Table \ref{table2} the group $N$ has a pointewisely fixed the hyperplane \mbox{section $t = 0$}. In type~$3$ there are no other fixed points and in type $1$ there is only one other fixed point $(0 : 0 : 0 : 1)$. Therefore by the Hurwitz formula
$$
K_{X / \CG_2}^2 = \frac{1}{2}\left( 2K_X \right)^2 = 6, \quad\quad K_{X / \CG_3}^2 = \frac{1}{3}\left( 3K_X \right)^2 = 9
$$
\noindent in types $1$ and $3$ respectively. The surface $X / N$ has at most du Val singularities. Therefore for the minimal resolution of singularities $\widetilde{X / N} \rightarrow X / N$ one has $K_{\widetilde{X / N}}^2 = K_{X / N}^2$. Thus for any $G / N$-MMP-reduction $Y$ of $X / N$ one has $K_Y^2 \geqslant 6$.

~

In type $2$ of Table \ref{table2} the group $\CG_2$ fixes pointwisely the lines $x = y = 0$ and $z = t = 0$. If one of these lines is tangent to $X$ at a point $t$ then in the neighbourhood of this point the group~$\CG_2$ acts as a reflection. Therefore there is a curve of $\CG_2$-fixed points passing through~$t$ contained in $X$. Thus one of the pointewisely fixed lines is contained in $X$. Therefore this line is defined over $\ka$ and can be $G$-equivariantly contracted. If both lines $x = y = 0$ and $z = t = 0$ intersect the surface~$X$ transversally then there are six \mbox{$\CG_2$-fixed} points on $X$ but by the Lefschetz fixed-point formula there are exactly four $\CG_2$-fixed points if $\CG_2$ does not have pointwisely fixed curves.

So the quotient $X / N$ is $G / N$-birationally equivalent to the quotient of del Pezzo surface of degree $4$ by a group of order $2$ having $4$ fixed points one of which is \mbox{$G \times \Gal \left( \kka / \ka \right)$-fixed}. By Lemma \ref{DP4i12} there exists a $G / N$-MMP-reduction $Y$ of the latter quotient such \mbox{that $K_Y^2 = 8$}.

~

In type $4$ of Table \ref{table2} the group $\CG_3$ fixes pointwisely the lines $x = y = 0$ and $z = t = 0$. These lines intersect $\XX$ given by
$$
x^3 + y^3 + z^3 + t^3 = 0
$$
\noindent at points $p_1$, $p_2$, $p_3$ and $q_1$, $q_2$, $q_3$ respectively. Let $C_{ij}$ be a line in $\Pro^3_{\kka}$ passing through $p_i$ and $q_j$.

Assume that $C_{ij}$ does not lie in $\XX$. For some integer $a$ the involution
$$
(x : y : z : t) \mapsto (\omega^{2a} z : t : \omega^a x : y)
$$
\noindent permutes points $p_i$ and $q_j$, thus the line $C_{ij}$ is invariant under the action of this involution. Therefore the line $C_{ij}$ cannot be tangent to $\XX$ at any of the points $p_i$ and $q_j$. Then the third point of intersection of $C_{ij}$ with $\XX$ is $\CG_3$-fixed. Thus there are three $\CG_3$-fixed points on $C_{ij}$ but this contradicts the fact that the action of $\CG_3$ is faithful on $C_{ij}$. So $C_{ij}$ lies in $\XX$ and $C_{ij}^2 = -1$.

Let $f: X \rightarrow X / N$ be the quotient morphism and
$$
\pi: \widetilde{X / N} \rightarrow X / N
$$
\noindent be the minimal resolution of singularities. Then $f(p_i)$ and $f(q_j)$ are singularities of type~$\frac{1}{3}(1,1)$. Thus $\pi^{-1}_*f(C_{ij})$ are $9$ disjoint $(-1)$-curves (see Table \ref{table1}). We can contract these curves and get a surface $Y$. One has
$$
K_Y^2 = K_{\widetilde{X / N}}^2 + 9 = K_{X / N}^2 + 9 - 6 \cdot \frac{1}{3} = \frac{1}{3}K_X^2 + 7 = 8.
$$

In type $6$ of Table \ref{table2} the group $\CG_5$ has two invariant lines $x = z = 0$ and $y = t = 0$ lying in $\XX$ given by the equation
$$
x^2 y + y^2 z + z^2 t + t^2 x = 0.
$$
\noindent One can $G$-equivariantly contract this pair and get a del Pezzo surface of degree $5$.

So the quotient $X / N$ is $G / N$-birationally equivalent to the quotient of del Pezzo surface of degree $5$ by a group of order $5$. By Theorem \ref{ratquot4} this quotient is $\ka$-rational so it is \mbox{$G / N$-birationally} equivalent to a surface $Y$ such that $K_Y^2 \geqslant 5$.
\end{proof}

\begin{corollary}
\label{DP3O6}
Let a finite group $G$ of order $6$ act on a del Pezzo surface $X$ of degree $3$. Then the surface $X / G$ is birationally equivalent to a surface $Y$ such that $K_Y^2 \geqslant 5$.
\end{corollary}

\begin{proof}
Let $N \subset G$ be the subgroup of order $3$. Then by Lemmas \ref{DP3cyclic} and \ref{DP3C35} the quotient $X / N$ is $G / N$-birationally equivalent to a surface $Z$ such that either $K_Z^2 \geqslant 5$ or $Z$ is a del Pezzo surface of degree $3$. There exists an MMP-reduction $Y$ of $X / G \approx Z / (G / N)$ such that $K_Y^2 \geqslant 5$ by Theorem \ref{ratquot4} and Lemma \ref{DP3cyclic} respectively.
\end{proof}

\begin{remark}
\label{DP3type3}
Note that for an element $g$ of type $3$ of Table \ref{table2} the quotient $\XX / \langle g \rangle$ is isomorphic to $\Pro^2_{\kka}$. Therefore one has
$$
\rho\left(\XX\right)^{\langle g \rangle} = \rho\left(\Pro^2_{\kka}\right) = 1.
$$
\end{remark}

To prove Theorem \ref{DP3} we need to list all possible automorphism groups of cubic surfaces.

\begin{theorem}[{cf. \cite[Subsection 6.5, Table 4]{DI1}}]
\label{DP3groupclass}
Let $\XX$ be a del Pezzo surface of degree $3$. Then one can choose homogeneous coordinates $x$, $y$, $z$, $t$ in $\Pro^3_{\kka}$ such that the equation of $\XX$ and the full automorphism group $\Aut(\XX)$ are presented in Table \ref{table3}, where $u$, $v$ and $\alpha$ are coefficients, and $H_3(3)$ is a group generated by the transformation
$$
(x : y : z : t) \mapsto (x : \omega y : \omega^2 z : t)
$$
\noindent and a cyclic permutation of $x$, $y$ and $z$.

\begin{table}
\caption{} \label{table3}

\begin{tabular}{|c|c|c|c|}
\hline
Type & Group & Equation  \\
\hline
$\mathrm{I}$ & $\CG_3^3 \rtimes \SG_4$ & $x^3 + y^3 + z^3 + t^3 = 0$ \\
\hline
$\mathrm{II}$ & $\SG_5$ & $x^2 y + y^2 z + z^2 t + t^2 x = 0$ \\
\hline
$\mathrm{III}$ & $H_3(3) \rtimes \CG_4$ & $x^3 + y^3 + z^3 + \alpha xyz + t^3 = 0$ \\
\hline
$\mathrm{IV}$ & $H_3(3) \rtimes \CG_2$ & $x^3 + y^3 + z^3 + \alpha xyz + t^3 = 0$ \\
\hline
$\mathrm{V}$ & $\SG_4$ & $t(x^2 + y^2 + z^2) + \alpha xyz + t^3 = 0$ \\
\hline
$\mathrm{VI}$ & $\SG_3 \times \CG_2$ & $x^3 + y^3 + \alpha zt(x + y) + z^3 + t^3 = 0$ \\
\hline
$\mathrm{VII}$ & $\CG_8$ & $x^3 + xy^2 + yz^2 + zt^2 = 0$ \\
\hline
$\mathrm{VIII}$ & $\SG_3$ & $x^3 + y^3 + zt(ux + vy) + z^3 + t^3 = 0$ \\
\hline
$\mathrm{IX}$ & $\CG_4$ & $x^3 + \alpha y^3 + xy^2 + yz^2 + zt^2 = 0$ \\
\hline
$\mathrm{X}$ & $\CG_2^2$ & $x^3 + y^3 + z^3 + \alpha xyz + t^2(x + y + uz) = 0$ \\
\hline
$\mathrm{XI}$ & $\CG_2$ & $x^3 + y^3 + z^3 + \alpha xyz + t^2(x + uy + vz) = 0$ \\
\hline
\end{tabular}

\end{table}

\end{theorem}

In the paper \cite{DI1} there is one more column in this table which contains conditions on the parameters. But we are interested only in the structure of the group and its action on $\Pro^3_{\kka}$ so we omit this column.

\begin{lemma}
\label{DP3V4}
Let a finite group $G$ act on a del Pezzo surface $X$ of degree $3$ and $N \cong \VG_4$ be a normal subgroup in $G$ such that nontrivial elements of $N$ act as in type $2$ of Table \ref{table2}. Then there exists a \mbox{$G / N$-MMP-reduction $Y$} of $X / N$ such that $K_Y^2 = 6$.
\end{lemma}

\begin{proof}

One can choose coordinates in $\Pro^3_{\kka}$ in which $\XX$ is given by the equation
$$
t(x^2 + y^2 + z^2) + \alpha xyz + t^3 = 0
$$
\noindent and nontrivial elements of $\VG_4$ switch signes of $t$ and one other variable. The set of fixed by nontrivial elements of group $\VG_4$ points consists of three pointwisely fixed lines
$$
x = t = 0, \qquad y = t = 0, \qquad z = t = 0
$$
\noindent lying in $\XX$ and six isolated fixed points $(1 : 0 : 0 : \pm \mathrm{i})$, $(0 : 1 : 0 : \pm \mathrm{i})$ and $(0 : 0 : 1 : \pm \mathrm{i})$. Thus the quotient $X / N$ is a singular del Pezzo surface with three $A_1$ singularities. By the Hurwitz formula
$$
K_{X / N}^2 = \frac{1}{4}\left(2K_X\right)^2 = 3.
$$
Let $q_1$, $q_2$ and $q_3$ be singular points of $X / N$. Consider the anticanonical embedding
$$
X / N \hookrightarrow \Pro^3_{\ka}.
$$
\noindent Denote by $C_{ij}$ a line in $\Pro^3_{\kka}$ passing through $q_i$ and $q_j$. Such a line contains two singular points on the surface $X / N$ of degree $3$, therefore all lines $C_{ij}$ lie in $X / N$. Moreover, one has
$$
K_{X / N} \cdot C_{ij} = -1.
$$
\noindent Thus we can resolve the singularities of $X / N$, then $G / N$-equivariantly contract the proper transforms of $C_{ij}$ and get a surface $Y$ with $K_Y^2 = 6$.

\end{proof}

\begin{lemma}
\label{DP3A5}
Let a finite group $G$ act on a del Pezzo surface $X$ of degree $3$ and $N \cong \AG_5$ be a normal subgroup in $G$. Then there exists a \mbox{$G / N$-MMP-reduction $Y$} of $X / N$ such that $K_Y^2 \geqslant 6$.
\end{lemma}

\begin{proof}

Consider the group $\AG_5$ acting on $\Pro^2_{\kka}$. By a direct computation one can show that each element $g$ of $\AG_5$ is a composition of two elements $h_1$ and $h_2$ of order two. Any element of order $2$ in $\mathrm{PGL_3}\left( \ka \right)$ has a line of fixed points. Thus each element $g$ of $\AG_5$ has an isolated fixed point $p$ which is the intersection point of the lines fixed by elements $h_1$ and $h_2$. The stabilizer group $N_p$ of $p$ acts on the tangent space of $\Pro^2_{\kka}$ at $p$. Therefore $N_p$ is a subgroup in $\AG_5$ and $\mathrm{GL}_2\left( \ka \right)$. Thus this is isomorphic to $\CG_2^2$, $\SG_3$ or $\DG_{10}$ if $\ord g$ is $2$, $3$ or $5$ respectively. The images of these group in $\mathrm{GL}_2\left( \ka \right)$ are generated by reflections.

Consider such a point $p$ for an element $g$ of order $5$. The stabilizer $N_p$ of $p$ is isomorphic to $\DG_{10}$ and $g$ acts in the tangent space of $\Pro^2_{\kka}$ at $p$ as $\operatorname{diag}(\xi_5, \xi_5^4)$. One can easily check that the action of $g$ in the tangent spaces of the two other fixed points is conjugate to~$\operatorname{diag}(\xi_5, \xi_5^2)$.

The group $\AG_5$ contains six subgroups isomorphic to $\CG_5$. Let $p_1$, \ldots, $p_6$ be fixed points of these subgroups whose stabilizers are isomorphic to $\DG_{10}$. Consider an $\AG_5$-equivariant blowup $\sigma: \XX \rightarrow \Pro^2_{\kka}$ of the points $p_1$, \ldots, $p_6$. The surface $\XX$ is a del Pezzo surface of degree $3$. From Theorem~\ref{DP3groupclass} one can see that there is a unique cubic surface with $\AG_5$ action on it. This cubic surface is called Clebsch cubic. We show that the action of $\AG_5$ on the Clebsch cubic is conjugate to the action of $\AG_5$ on the blowup $\XX \rightarrow \Pro^2_{\kka}$ at the six points $p_i$.

We use the notation of Remark \ref{DP_1curves}. Let $g_2$, $g_3$ and $g_5$ be elements in $\AG_5$ of order $2$, $3$ and $5$ respectively.

The stabilizer of any point $p_i$ in $\AG_5$ is isomorphic to $\DG_{10}$. Therefore there are $5$ lines passing through the point $p_i$ that are pointewisely fixed by an element of order $2$ in $\AG_5$. But in $\AG_5$ there are only $15$ elements of order~$2$. Thus each element of order $2$ fixes pointwisely a line passing through a pair of points $p_i$ and $p_j$ on $\Pro^2_{\kka}$ and fixes pointwisely a $(-1)$-curve $L_{ij}$ on $\XX$. By the Lefschetz fixed-point formula the element $g_2$ has three isolated fixed points. Two of them are $E_i \cap Q_j$ and $Q_i \cap E_j$ and the third is the preimage of the isolated fixed point $p$ of $g_2$ on $\Pro^2_{\kka}$. In the tangent space of $\XX$ at $p$ the stabilizer group $N_p$ of $p$ acts as $\CG_2^2$ generated by reflections.

An element $g_3$ does not have any invariant $(-1)$-curve. Therefore it cannot have curves of fixed points. Thus the action of $g_3$ on $\Pro^2_{\kka}$ is conjugate to $\operatorname{diag} (1, \omega, \omega^2)$. Hence on the surface $\XX$ the element $g_3$ has three isolated points and acts in the tangent spaces of $\Pro^2_{\kka}$ at these points as $\operatorname{diag} (\omega, \omega^2)$. These points cannot be points of the blowup since the stabilizer of any point $p_i$ in $\AG_5$ is $\DG_{10}$. Therefore there are three $g_3$-fixed points on $\XX$ and in the tangent spaces of $\XX$ at these points as $\operatorname{diag} (\omega, \omega^2)$. Two of these points do not lie on $(-1)$-curves and the third one is a point of intersection of three $(-1)$-curves.

An element $g_5$ has three fixed points on $\Pro^2_{\kka}$: namely $p_k$ for some $k \in \{1, 2, 3, 4, 5, 6 \}$ and two points in whose tangent space the action of $g_5$ is conjugate to $\operatorname{diag}(\xi_5, \xi_5^2)$. The $(-1)$-curve $Q_k$ is $g_5$-invariant thus the quadric $\sigma(Q_k)$ passes through two $g_5$-fixed points different from $p_k$. Therefore the element $g_5$ has four fixed points on $\XX$, two of them lie on $E_k$ and two on~$Q_k$. The element~$g_5$ acts in the tangent spaces of $\XX$ at $g_5$-fixed points as $\operatorname{diag}(\xi_5, \xi_5^2)$.

Let $f: X \rightarrow X / N$ be the quotient morhism,
$$
\pi: \widetilde{X / N} \rightarrow X / N
$$
\noindent be the minimal resolution of singularities, and put $E = f(E_i)$, $Q = f(Q_j)$. There are four singular points on $X / N$: two singular points of type $\frac{1}{5}(1,2)$ lie on the curves $E$ and $Q$ respectively, one singular point of type $A_1$ is the intersection point $E \cap Q$ and one singular point of type $A_2$ lies neither on $E$ nor on $Q$. We have (see Table \ref{table1}):
$$
K_{\widetilde{X / N}}^2 = K_{X / N}^2 - \frac{4}{5} = \frac{1}{60}(6K_X)^2 - \frac{4}{5} = 1,
$$
$$
\rho(\widetilde{X / N})^{G / N} \geqslant \rho(X / N)^{G / N} + 4 = \rho(X)^G + 4 \geqslant 5,
$$
$$
\pi^{-1}_*(E)^2 = E^2 - \frac{1}{2} - \frac{2}{5} = \frac{1}{60}\left(\sum \limits_{i=1}^6 E_i \right)^2 - \frac{9}{10} = -1,
$$
$$
\pi^{-1}_*(Q)^2 = Q^2 - \frac{1}{2} - \frac{2}{5} = \frac{1}{60}\left(\sum \limits_{i=1}^6 Q_i \right)^2 - \frac{9}{10} = -1.
$$
Thus we can $G / N$-equivariantly contract curves $\pi^{-1}_*(E)$ and $\pi^{-1}_*(Q)$. We obtain a surface $Z$ such that $K_Z^2 = 3$ and $\rho(Z)^{G / N} \geqslant 4$. By Corollary \ref{piccrit} there exists a $G / N$-minimal model $Y$ of $Z$ such that $K_Y^2 \geqslant 6$.

\end{proof}

Now we prove Theorem \ref{DP3}.

\begin{proof}[Proof of Theorem \ref{DP3}]
We consider each case of Theorem \ref{DP3groupclass} and show that if the group $G$ is not trivial and is not conjugate to $\CG_3$ acting as in type $5$ of Table \ref{table2} then there exists a \mbox{$G / N$-MMP-reduction $Y$} of $X / N$ such that $K_Y^2 \geqslant 5$.

In case I the group $\Aut(\XX)$ is $\CG_3^3 \rtimes \SG_4$. The group $\CG_3^3$ is a diagonal subgroup of $\mathrm{PGL}_4\left( \ka \right)$ and $\SG_4$ permutes coordinates. We consider a normal subgroup $H = G \cap \CG_3^3$. If there is an element of type $3$ in this subgroup then we consider a group $N \subset H$ generated by elements of type $3$. Then the group $N$ is normal and one of the following possibilities holds:

\begin{itemize}

\item If $N$ is generated by one element of type $3$ then $N \cong \CG_3$, the quotient $X / N$ is smooth and by the Hurwitz formula one has
$$
K_{X / N}^2 = \frac{1}{3} \left( 3K_X \right)^2 = 9.
$$

\item If $N$ is generated by two elements of type $3$ then $N \cong \CG_3^2$, the quotient $X / N$ has only one singular point of type $\frac{1}{3}(1,1)$ and by the Hurwitz formula one has
$$
K_{X / N}^2 = \frac{1}{9} \left( 5K_X \right)^2 = \frac{25}{3}.
$$

\item If $N$ is generated by three elements of type $3$ then $N \cong \CG_3^3$, the quotient $X / N$ is smooth and by the Hurwitz formula one has
$$
K_{X / N}^2 = \frac{1}{27} \left( 9K_X \right)^2 = 9.
$$

\end{itemize}

For any $G / N$-MMP-reduction $Y$ of $X / N$ one has $K_Y^2 \geqslant 8$.

If the group $H$ does not contain elements of type $3$ then either $H$ is trivial or $H$ is isomorphic to $\CG_3$ generated by an element of type $4$ or $5$ or $\CG_3^2$ generated by the elements $\operatorname{diag}(1, 1, \omega, \omega)$ and $\operatorname{diag}(1, \omega, 1, \omega)$.

In the last case the group $G$ is a subgroup of $\CG_3^2 \rtimes \DG_8$ where $\DG_8 = \langle (1243), (23) \rangle$. If $G$ contains the element $(14)(23)$ then the group $N = \langle (14)(23) \rangle$ is normal in $G$ and there exists a \mbox{$G / N$-MMP-reduction $Y$} of $X / N$ such that $K_Y^2 \geqslant 5$ by Lemma \ref{DP3cyclic}. Otherwise the group $G$ is isomorphic to $\CG_3^2 \rtimes \CG_2$ or $\CG_3^2$.

If $G$ contains a normal subgroup $N \cong \CG_3$ generated by element of type $4$ then there exists a \mbox{$G / N$-MMP-reduction $Y$} of $X / N$ such that $K_Y^2 \geqslant 5$ by Lemma \ref{DP3cyclic}. Otherwise $G$ is conjugate to $\CG_3^2 \rtimes \CG_2$ where $\CG_2$ is either $\langle (23) \rangle$ or $\langle (14) \rangle$. In this case $G$ contains a normal subgroup $N \cong \CG_3$ generated by an element of type $5$. By Lemma \ref{DP3C35} the quotient $X / N$ is \mbox{$G / N$-birationally} equivalent to a del Pezzo surface $Z$ of degree $3$ and the \mbox{group $G / N \cong \SG_3$}. The quotient $Z / \SG_3$ is $\ka$-birationally equivalent to a surface $Y$ with $K_Y^2 \geqslant 5$ by Corollary~\ref{DP3O6}.

If the group $H \cong \CG_3$ is generated by an element of type $4$ then there exists a \mbox{$G / H$-MMP-reduction $Y$} of $X / H$ such that $K_Y^2 \geqslant 5$ by Lemma \ref{DP3cyclic}.

If the group $H \cong \CG_3$ is generated by an element of type $5$ then $G \subset \SG_3 \times \CG_2$ and the quotient $X / H$ is $G / H$-birationally equivalent to a del Pezzo surface $Z$ of degree $3$ by Lemma \ref{DP3C35}. Thus if $G / H$ is nontrivial then it contains a subgroup $N$ of order $2$. There exists a \mbox{$G / N$-MMP-reduction $Y$} of $X / N$ such that $K_Y^2 \geqslant 5$ by Lemma \ref{DP3cyclic}.

If the group $H$ is trivial then $G$ is a subgroup of $\SG_4$. Then the group $G$ contains a normal subgroup $N$ isomorphic to $\CG_2$, $\CG_3$ or $\VG_4$ by Lemma \ref{S5subgroups}. If $N \cong \CG_2$ or $N \cong \VG_4$ then there exists a \mbox{$G / N$-MMP-reduction $Y$} of $X / N$ such that $K_Y^2 \geqslant 5$ by Lemma \ref{DP3cyclic} or Lemma \ref{DP3V4} respectively. If $N \cong \CG_3$ then either $G$ is generated by an element of type $5$ or $G \cong \SG_3$ and there exists an \mbox{MMP-reduction $Y$} of $X / G$ such that $K_Y^2 \geqslant 5$ by Corollary~\ref{DP3O6}.

~

In case II the group $G$ contains a normal subgroup $N$ isomorphic to $\CG_2$, $\CG_3$, $\VG_4$, $\CG_5$ or $\AG_5$ by Lemma \ref{S5subgroups}. If $N$ is not isomorphic to $\CG_3$ then there exists a \mbox{$G / N$-MMP-reduction $Y$} of $X / N$ such that $K_Y^2 \geqslant 5$ by Lemma \ref{DP3cyclic}, Lemma \ref{DP3V4} or Lemma~\ref{DP3A5}. Otherwise $N \cong \CG_3$ and $G$ is a subgroup of $\SG_3 \times \CG_2$. Subgroups of this group are considered in case I.

~

In cases III and IV let us consider the group $H = G \cap H_3(3)$. If $\ord H > 3$ then $H$ contains a normal subgroup $N$ generated by an element of order $3$ acting as in type $3$ of Table \ref{table2}. The group $N$ is normal in $G$ and there exists a \mbox{$G / N$-MMP-reduction $Y$} of $X / N$ such that $K_Y^2 \geqslant 5$ by Lemma \ref{DP3cyclic}.

If $H$ is generated by an element of type $5$ and $G / H$ is not trivial then $G \cong \SG_3$ and there exists an \mbox{MMP-reduction $Y$} of $X / G$ such that $K_Y^2 \geqslant 5$ by Corollary~\ref{DP3O6}.

If $H$ is trivial and $G$ is not trivial then $G \subset \CG_4$ contains a normal subgroup $N$ of order two and there exists an \mbox{$G / N$-MMP-reduction $Y$} of $X / N$ such that $K_Y^2 \geqslant 5$ by Lemma~\ref{DP3cyclic}.

~

In case VII if $G$ is not trivial then $G$ contains a normal subgroup $N$ of order two and there exists an \mbox{$G / N$-MMP-reduction $Y$} of $X / N$ such that $K_Y^2 \geqslant 5$ by Lemma \ref{DP3cyclic}.

~

In the other cases the group $G$ is conjugate in $\mathrm{PGL_4}(\kka)$ to a subgroup of $\SG_5$. All these possibilities were considered in case II.

In all cases one has $Y(\ka) \ne \varnothing$ since $X(\ka) \ne \varnothing$. Thus
$$
Y / (G / N) \approx X / G
$$
\noindent is $\ka$-rational by Theorem \ref{ratquot4}.
\end{proof}

\section{Minimality conditions}

Let $X$ be a cubic surface in $\Pro^3_{\ka}$ and let a group $G \cong \CG_3$ act on $X$ as in type $5$ of Table~\ref{table2}. In this section we find some conditions for the action of the Galois group $\Gal \left(\kka / \ka \right)$ on the set of $(-1)$-curves under which the surface $X$ is $\ka$-rational and $X / G$ is not $\ka$-rational.

Throughout this section we use the notation of Remark \ref{DP_1curves}. Let $\Gamma$ be the image of the Galois group $\Gal\left(\kka / \ka\right)$ in the Weyl group $W(E_6)$ acting on $\Pic(\XX)$  (see \cite[\S IV.3]{Man74}). The group $\Gamma$ effectively acts on the set of $(-1)$-curves on $X$. The group $W(E_6)$ contains a subgroup $\SG_6$ acting in the following way: for $\sigma \in \SG_6$ one has
$$
\sigma E_i = E_{\sigma(i)}, \qquad \sigma L_{ij} = L_{\sigma(i)\sigma(j)}, \qquad \sigma Q_i = Q_{\sigma(i)}.
$$

\begin{lemma}
\label{PicC35}
The image of the group $G$ in the Weyl group~$W(E_6)$ is conjugate to~$\langle (123)(456) \rangle$.
\end{lemma}

\begin{proof}

The order of the Weyl group $W(E_6)$ is equal to $51840 = 2^7 \cdot 3^4 \cdot 5$. By Sylow theorem all groups of order $81$ are conjugate in $W(E_6)$. The group of order $81$ acts on the Fermat cubic (see Table \ref{table3})
$$
x^3 + y^3 + z^3 + t^3 = 0.
$$
Thus any element of order $3$ in $W(E_6)$ has type $3$, $4$ or $5$ of Table \ref{table2}.

For any element $g$ of type $3$ of Table~\ref{table2} one has $\rho(\XX)^{\langle g \rangle} = 1$ by Remark \ref{DP3type3}. In the subgroup $\SG_6 \subset W(E_6)$ there is an element $(123)(456)$ which has type~$4$ or $5$ of Table~\ref{table2} since
$$
\rho(\XX)^{\langle (123)(456) \rangle} > 1.
$$
\noindent But an element of type $4$ has invariant $(-1)$-curves (see the proof of Lemma \ref{DP3cyclic}) and the element $(123)(456)$ does not have invariant $(-1)$-curves. Therefore the action of the group $G$ is conjugate to $\langle (123)(456) \rangle$ in the group $W(E_6)$.

\end{proof}

\begin{remark}
\label{Man}

An alternative proof of Lemma \ref{PicC35} is the following. One can look at Table $1$ in \cite[Chapter IV, $\S5$]{Man74} and see that conjugacy classes of elements of order $3$ in the group $W(E_6)$ correspond to the rows $3$, $18$ and $22$ of this table. But in the eighth column of the table one can see that the element $g$ correspond to the $18$-th row. In the ninth column one can see that $g$ is conjugate to~$(123)(456)$ in $W(E_6)$. Also one can see that the order of the centralizer of~$G$ in $W(E_6)$ is $108$.

\end{remark}

From now on we can assume that the group $G$ acts on the set of $(-1)$-curves on $X$ as~$\langle (123)(456) \rangle$. The Galois group $\Gal \left( \kka / \ka \right)$ commutes with the group $G$. Therefore to describe possibilities for the group $\Gamma$ we should find the centralizer of $G$ in $W(E_6)$.

\begin{lemma}
\label{S6center}
The centralizer of $G = \langle (123)(456) \rangle$ in $\SG_6$ is a subgroup $\CG_3^2 \rtimes \CG_2$ generated by $a = (123)$, $b = (456)$ and $c = (14)(25)(36)$.
\end{lemma}

\begin{proof}
Note that in the group $\SG_6$ there are
$$
\frac{6!}{3! \cdot 3! \cdot 2} \cdot 4 = 40
$$
\noindent elements conjugate to $(123)(456)$. Therefore the order of the centralizer of \mbox{$G = \langle (123)(456) \rangle$} is equal to $18$.

The elements $a$, $b$ and $c$ obviously commute with the element $(123)(456)$ and the group $\CG_3^2 \rtimes \CG_2 = \langle a, b, c \rangle$ has order $18$. Thus the centralizer of $G = \langle (123)(456) \rangle$ in $\SG_6$ is a subgroup $\CG_3^2 \rtimes \CG_2 = \langle a, b, c \rangle$.

\end{proof}

Note that the group $G$ has exactly three orbits that consist of \mbox{$(-1)$-curves} meeting each other: $\{ L_{14}, L_{25}, L_{36} \}$, $\{L_{15}, L_{26}, L_{34}\}$ and $\{L_{16}, L_{24}, L_{35}\}$. The other orbits consist of disjoint $(-1)$-curves. Therefore the set of nine $(-1)$-curves
$$
L_{ij},\,\, \text{where}\,\, i \in \{1,2,3\}\,\, \text{and} \,\,j \in \{4,5,6\}
$$
\noindent is invariant under the action of centralizer of $G$ in $W(E_6)$. The group
$$
\CG_3^2 \rtimes \CG_2 \cong \langle a, b, c \rangle
$$
\noindent can realize any permutation of this set of $(-1)$-curves that preserves the intersection form. Therefore to find the centralizer of $G$ in $W(E_6)$ we should find a subgroup in $W(E_6)$ acting trivially on the set of nine $(-1)$-curves $L_{ij}$, where $i \in \{1,2,3\}$ and $j \in \{4,5,6\}$.

\begin{lemma}
\label{W6fixL}
The subgroup of $W(E_6)$ fixing each of the nine $(-1)$-curves $L_{ij}$, where $i \in \{1,2,3\}$ and $j \in \{4,5,6\}$, is a group $\SG_3$ generated by elements $r$ and $s$ of order $3$ and $2$ respectively acting on the set of $(-1)$-curves in the following way:
$$
s\left(E_i\right) = Q_i, \qquad s\left(Q_i\right) = E_i, \qquad s\left(L_{ij}\right) = L_{ij},
$$
$$
r\left(E_i\right) = Q_i\,\, \text{if}\,\, i \in \{1,2,3\}, \qquad r^2\left(E_i\right) = Q_i\,\, \text{if}\,\, i \in \{4,5,6\},
$$
$$
r\left(E_i\right) = L_{jk}\,\, \text{if}\,\, i \in \{4,5,6\}\,\, \text{and}\,\, j,k \in \{4,5,6\}\,\, \text{differ from}\,\, i,
$$
$$
r^2\left(E_i\right) = L_{jk}\,\, \text{if}\,\, i \in \{1,2,3\}\,\, \text{and}\,\, j,k \in \{1,2,3\}\,\, \text{differ from}\,\, i.
$$
\end{lemma}

\begin{proof}
Let us consider the $(-1)$-curve $E_1$. Since $L_{14}$, $L_{15}$ and $L_{16}$ are invariant, the image of $E_1$ can be only $E_1$, $L_{23}$ or $Q_1$. The action of the group on these three $(-1)$-curves defines the whole action of the group $\langle r, s \rangle$ on the set of $(-1)$-curves. The group $\SG_3 = \langle r, s \rangle$ fixes all the nine $(-1)$-curves $L_{ij}$, where $i \in \{1,2,3\}$ and $j \in \{4,5,6\}$, and permutes $E_1$, $L_{23}$ and $Q_1$ in all possible ways.
\end{proof}

\begin{proposition}
\label{W6center}
The centralizer of $G = \langle (123)(456) \rangle$ in $W(E_6)$ is a subgroup
$$
H \cong \left( \CG_3^2 \rtimes \CG_2 \right) \times \SG_3
$$
\noindent where the first factor is generated by $a$, $b$ and $cs$, and the second factor is generated by $r$ and $s$.
\end{proposition}

\begin{proof}
By Lemmas \ref{S6center} and \ref{W6fixL} the centralizer of $G$ in $W(E_6)$ is generated by the subgroups \mbox{$\CG_3^2 \rtimes \CG_2 = \langle a, b, c \rangle$} and $\SG_3 = \langle r, s \rangle$. Obviously, the elements $a$, $b$, $cs$, $r$ and $s$ generate this group. One can easily check that $a$, $b$, $cs$ and $r$, $s$ pairwisely commute.
\end{proof}

By Remark \ref{DP3C35rat} if the quotient of $X$ is not $\ka$-rational then $\rho(X)^G = 1$. Moreover, if $\rho(X) = 1$ then $X$ is not $\ka$-rational by Theorem \ref{ratcrit}. Thus to construct non-$\ka$-rational quotient of $\ka$-rational cubic surface we should find all possibilities of the Galois group $\Gamma$ such that $\rho(X) > 1$ and $\rho(X)^G = 1$.

The group $\Gamma$ is a subgroup of the group $H \cong \left( \CG_3^2 \rtimes \CG_2 \right) \times \SG_3$ where the first factor is generated by $a$, $b$ and $cs$, and the second factor is generated by $r$ and $s$. We denote the projection on the first factor
$$
H \cong \left( \CG_3^2 \rtimes \CG_2 \right) \times \SG_3 \twoheadrightarrow \CG_3^2 \rtimes \CG_2
$$
\noindent by $\pi_1$, and the projection on the second factor
$$
H \cong \left( \CG_3^2 \rtimes \CG_2 \right) \times \SG_3 \twoheadrightarrow \SG_3
$$
\noindent by $\pi_2$.

\begin{lemma}
\label{Galnonmin1}
If $\pi_2(\Gamma)$ is trivial and $\pi_1(\Gamma) \subset \CG_3^2$ or $\pi_2(\Gamma) = \langle s \rangle$ but $s \notin \Gamma$ then $\rho(X)^G > 1$.
\end{lemma}

\begin{proof}
In these cases the group $\Gamma$ is a subgroup of the group $\langle a, b, c \rangle$. Therefore the groups~$\Gamma$ and $G \cong \CG_3$ preserve $\sum \limits_{i=1}^6 E_i$ and $\rho(X)^G > 1$.
\end{proof}

\begin{lemma}
\label{Galnonmin2}
If $\pi_2(\Gamma) = \langle s \rangle$ and $\pi_1(\Gamma) \subset \CG_3^2$ then $\rho(X)^G > 1$.
\end{lemma}

\begin{proof}
In this case the groups $\Gamma$ and $G$ preserve the divisor $\sum \limits_{i=1}^3 E_i - \sum \limits_{i=4}^6 E_i$. Therefore one has $\rho(X)^G > 1$.
\end{proof}

\begin{lemma}
\label{Galmin1}
One has $\rho(X)^{\langle abcs \rangle} = \rho(X)^{\langle a^2b, cs \rangle} = 1$.
\end{lemma}

\begin{proof}
One has $(abcs)^4 = ab$ and $(abcs)^3 = cs$.

Note that the groups $\Pic(X)^{\langle ab \rangle}$ and $\Pic(X)^{\langle a^2b \rangle}$ are generated by $K_X$, $\sum \limits_{i=1}^3 E_i$ and~$\sum \limits_{i=4}^6 E_i$. One can check that $cs$-invariants in $\Pic(X)^{\langle ab \rangle}$ and $\Pic(X)^{\langle a^2b \rangle}$ are generated by~$K_X$. 
\end{proof}

\begin{corollary}
\label{Galmin2}
Suppose that $\pi_1(\Gamma)$ contains the element $cs$ and at least one element of order $3$ and $cs \in \Gamma$ then $\rho(X) = 1$.
\end{corollary}

\begin{lemma}
\label{Galmin3}
One has $\rho(X)^{\langle abr \rangle} = \rho(X)^{\langle a^2br \rangle} = 1$.
\end{lemma}

\begin{proof}
Note that any $abr$ and $a^2br$ orbit of a $(-1)$-curve consists of three $(-1)$-curves meeting each other. Therefore these elements has type $3$ of Table \ref{table2}. Thus by Remark \ref{DP3type3} one has
$$
\rho(X)^{\langle abr \rangle} = \rho(X)^{\langle a^2br \rangle} = 1.
$$
\end{proof}

\begin{corollary}
\label{Galmin4}
If $r \in \pi_2(\Gamma)$, and $\pi_1(\Gamma)$ contains $ab$ or $a^2b$ then $\rho(X) = 1$.
\end{corollary}

As a result of all the previous lemmas we have the following proposition.

\begin{proposition}
\label{GalClass}
Let $X$ be a del Pezzo surface of degree $3$ and $G \cong \CG_3$ be a group acting as in type $5$ of Table \ref{table2}. Let $\Gamma$ be the image of the Galois group $\Gal\left(\kka / \ka\right)$ in the Weyl group $W(E_6)$. If $\rho(X) > 1$ and $\rho(X)^G = 1$ then we have the following possibilities for $\Gamma$ up to conjugation:

\begin{enumerate}

\item $\Gamma = \langle cs \rangle \cong \CG_2$;

\item $\Gamma = \langle c, s \rangle \cong \CG_2^2$;

\item $\Gamma = \langle r \rangle \cong \CG_3$;

\item $\Gamma = \langle ar \rangle \cong \CG_3$;

\item $\Gamma = \langle a, r \rangle \cong \CG_3^2$;

\item $\Gamma = \langle cs, r \rangle \cong \CG_6$;

\item $\Gamma = \langle r, s \rangle \cong \SG_3$;

\item $\Gamma = \langle a, r, s \rangle \cong \CG_3 \times \SG_3$;

\item $\Gamma = \langle r, c \rangle \cong \SG_3$;

\item $\Gamma = \langle r, c, s \rangle \cong \SG_3 \times \CG_2$.

\end{enumerate}

\end{proposition}

\begin{proof}

At first we show that in all other cases one has either $\rho(X) = 1$ or $\rho(X)^G > 1$.

If $r \in \pi_2(\Gamma)$ then if $\pi_1(\Gamma)$ contains an element $ab$ or $a^2b$ then $\rho(X) = 1$ by Corollary~\ref{Galmin4}. Therefore in this case $\pi_1(\Gamma)$ should be trivial or conjugate to $\langle a \rangle$ or $\langle cs \rangle$. These possibilities correspond to cases (3)--(10) of the proposition.

Now we can assume that $r \notin \pi_2(\Gamma)$. If $\pi_1(\Gamma)$ contains the element $cs$ and at least one element of order $3$ then by Corollary \ref{Galmin2} one has $\rho(X) = 1$ in all cases except $\Gamma = \langle ab, c \rangle$, $\Gamma = \langle a^2b, c \rangle$ and $\Gamma = \langle a, b, c \rangle$. In the last three cases we have $\rho(X)^G > 1$ by Lemma \ref{Galnonmin1}. If $cs \notin \pi_1(\Gamma)$ then $\rho(X)^G > 1$ by Lemmas \ref{Galnonmin1} and \ref{Galnonmin2}. Therefore $\pi_1(\Gamma) = \langle cs \rangle$. This possibility corresponds to cases (1) and (2) of the proposition.

Now we show that in all these cases one has $\rho(X) > 1$ and~$\rho(X)^G = 1$.

In cases (1), (2), (6), (9) and (10) the $(-1)$-curves $L_{14}$, $L_{25}$ and $L_{36}$ are $\Gamma$-invariant. Therefore $X$ is not minimal and $\rho(X) > 1$. In cases (3), (4), (5), (7) and (8) the triple of disjoint \mbox{$(-1)$-curves} $L_{14}$, $L_{24}$ and $L_{34}$ is $\Gamma$-invariant. Therefore $X$ is not minimal and~$\rho(X) > 1$.

In cases (1) and (2) one has $\rho(X)^G = 1$ by Lemma \ref{Galmin1}. In the other cases one has $\rho(X)^G = 1$ by Lemma \ref{Galmin3}.

\end{proof}

Note that if one can contract a $(-1)$-curve defined over $\ka$ and $\rho(X) = 2$ or $\rho(X) = 3$ then the obtained del Pezzo surface either is not mininal or can be minimal del Pezzo surface of degree $4$. So for each case of Proposition \ref{GalClass} we should check whether the surface $X$ is $\ka$-rational.

\begin{lemma}
\label{Galnonrat1}
If the Galois group $\Gamma$ contains the element $cs$ then $X$ is not $\ka$-rational.
\end{lemma}

\begin{proof}
Note that the $(-1)$-curves $L_{14}$, $L_{25}$ and $L_{36}$ are $cs$-invariant and the other \mbox{$(-1)$-curves} form $cs$-invariant pairs of $(-1)$-curves which are not disjoint. The curves $L_{14}$, $L_{25}$ and $L_{36}$ meet each other. Therefore we can contract no more than one $(-1)$-curve and $X$ is not $\ka$-rational by Theorem \ref{ratcrit}.
\end{proof}

\begin{lemma}
\label{Galnonrat2}
If the Galois group $\Gamma$ contains the elements $c$ and $r$ then $X$ is not $\ka$-rational.
\end{lemma}

\begin{proof}
Note that if a $\langle c , r \rangle$-orbit contains $E_i$ then it contains $Q_j$ with $i \ne j$. Therefore we cannot contract any of these orbits. Also we cannot contract $\langle c , r \rangle$-invariant pairs $L_{15}$ and $L_{24}$, $L_{16}$ and $L_{34}$, $L_{26}$ and $L_{35}$. The $(-1)$-curves $L_{14}$, $L_{25}$ and $L_{36}$ are $\langle c , r \rangle$-invariant and meet each other. Therefore we cannot contract more than one $(-1)$-curve and $X$ is not $\ka$-rational by Theorem \ref{ratcrit}.
\end{proof}

\begin{lemma}
\label{Galrat}
If the Galois group $\Gamma$ is contained in $\langle a, r, s \rangle$ then $X$ is $\ka$-rational.
\end{lemma}

\begin{proof}
We can contract $(-1)$-curves $E_4$, $L_{56}$ and $Q_4$, and get a del Pezzo surface of degree~$6$ which is $\ka$-rational by Theorem \ref{ratcrit}.
\end{proof}

Now we can prove the following theorem.

\begin{theorem}
\label{Galcrit}
Let $X$ be a $\ka$-rational del Pezzo surface of degree $3$ and $G \cong \CG_3$ be a group acting as in type $5$ of Table \ref{table2}. Let $\Gamma$ be the image of the Galois group $\Gal\left(\kka / \ka\right)$ in the Weyl group $W(E_6)$. If $X / G$ is not $\ka$-rational then we have the following possibilities for $\Gamma$ up to conjugation:

\begin{enumerate}

\item $\Gamma = \langle r \rangle \cong \CG_3$;

\item $\Gamma = \langle ar \rangle \cong \CG_3$;

\item $\Gamma = \langle a, r \rangle \cong \CG_3^2$;

\item $\Gamma = \langle r, s \rangle \cong \SG_3$;

\item $\Gamma = \langle a, r, s \rangle \cong \CG_3 \times \SG_3$.

\end{enumerate}

\end{theorem}

\begin{proof}

For cases (1), (2), (6), (10) of Proposition \ref{GalClass} the Galois group $\Gamma$ contains the element $cs$. Therefore $X$ is not $\ka$-rational by Lemma \ref{Galnonrat1}. For case (9) of Proposition~\ref{GalClass} the surface $X$ is not $\ka$-rational by Lemma \ref{Galnonrat2}.

For cases (3), (4), (5), (7), (8) of Proposition \ref{GalClass} the Galois group $\Gamma$ is contained in $\langle a, r, s \rangle \cong \CG_3 \times \SG_3$. Therefore $X$ is $\ka$-rational by Lemma \ref{Galrat}.

\end{proof}

\section{Geometric interpretation}

In this section we give geometric interpretation of the actions of elements in the Galois group $\Gamma$ considered in Section $4$.

For convenience we assume that the field $\ka$ contains $\omega$. Therefore we can choose homogeneous coordinates in $\Pro^3_{\ka}$ such that the group $G$ acts as
$$
(x : y : z : t) \mapsto (x : y : \omega z : \omega^2 t)
$$
\noindent on the cubic surface $X$ given by the equation
\begin{equation}
\label{Eqcubic}
P(x : y) + zt(ux + vy) + z^3 + \alpha t^3 = 0
\end{equation}
\noindent where $P(x : y)$ is a homogeneous polynomial of degree $3$.

Let us consider a line $x = y = 0$. This line intersects $X$ in three points $e_1$, $e_2$ and $e_3$ given by the equation
\begin{equation}
\label{EqEckardt}
z^3 + \alpha t^3 = 0.
\end{equation}

\begin{definition}
\label{Eckardtpt}
A point $p$ on a cubic surface is called an \textit{Eckardt point} if there are three $(-1)$-curves passing through $p$.
\end{definition}

\begin{lemma}
\label{X3Eckardt}
The points $e_1$, $e_2$ and $e_3$ are Eckardt points.
\end{lemma}

\begin{proof}
The surface $X$ is given by equation \eqref{Eqcubic} in~$\Pro^3_{\ka}$. In coordinates $x$, $y$, $z$, $t$ the points $e_1$, $e_2$ and $e_3$ are $(0 : 0 : -\sqrt[3]{\alpha} : 1)$, $(0 : 0 : - \omega \sqrt[3]{\alpha} : 1)$ and $(0 : 0 : - \omega^2 \sqrt[3]{\alpha} : 1)$. Consider the tangent plane at the point $e_1$. Its equation is
$$
u\sqrt[3]{\alpha}x + v\sqrt[3]{\alpha}y = 3(\sqrt[3]{\alpha^2}z + \alpha t).
$$
\noindent We have
$$
zt(ux + vy) + z^3 + \alpha t^3 = 3zt(\sqrt[3]{\alpha}z + \sqrt[3]{\alpha^2}t) + z^3 + \alpha t^3=
$$
$$
 = (z + \sqrt[3]{\alpha}t)^3 = \left( \frac{\sqrt[3]{\alpha^{-1}}ux + \sqrt[3]{\alpha^{-1}}vy}{3} \right)^3.
$$
So that equation \eqref{Eqcubic} can be rewritten as
\begin{equation}
\label{Eqtangent}
P(x : y) + \frac{(ux + vy)^3}{27\alpha} = 0.
\end{equation}

The last equation has three roots $(\lambda_1 : \mu_1)$, $(\lambda_2 : \mu_2)$ and $(\lambda_3 : \mu_3)$. The three \mbox{$(-1)$-curves} passing through the point $e_1$ are given by
$$
u\sqrt[3]{\alpha}x + v\sqrt[3]{\alpha}y = 3(\sqrt[3]{\alpha^2}z + \alpha t)
$$
\noindent and $\mu_i x = \lambda_i y$. Similarly we can show that the three $(-1)$-curves passing through $e_2$ are given by
$$
u\sqrt[3]{\alpha}x + v\sqrt[3]{\alpha}y = 3(\omega \sqrt[3]{\alpha^2}z + \omega^2 \alpha t)
$$
\noindent and $\mu_i x = \lambda_i y$, and the three $(-1)$-curves passing through $e_3$ are given by
$$
u\sqrt[3]{\alpha}x + v\sqrt[3]{\alpha}y = 3(\omega^2 \sqrt[3]{\alpha^2}z + \omega \alpha t)
$$
\noindent and $\mu_i x = \lambda_i y$.
\end{proof}

\begin{remark}
\label{ExpNot}
Applying explicit equations given in the proof of Lemma \ref{X3Eckardt} one can see that the $G$-orbit of any $(-1)$-curve passing through a point $e_i$ consists of three $(-1)$-curves meeting each other at a point. The image of $G$ in the Weyl group $W(E_6)$ is conjugate \mbox{to $\langle (123)(456) \rangle$} by Lemma \ref{PicC35}. Therefore nine curves passing through the Eckardt points $e_i$ are $L_{ij}$ where $i \in \{1, 2, 3\}$ and $j \in \{4, 5, 6\}$. We can assume that the curves $L_{14}$, $L_{26}$ and $L_{35}$ pass through $e_1$, the curves $L_{16}$, $L_{25}$ and $L_{34}$ pass through $e_2$ and the curves $L_{15}$, $L_{24}$ and $L_{36}$ pass through $e_3$.
\end{remark}

Now we give explicit geometric interpretation of the action of the group $\pi_1(\Gamma)$.

\begin{lemma}
\label{P1case}
Let $X$ be a cubic surface given by equation \eqref{Eqcubic} and $\Gamma$ be the image of the Galois group $\Gal\left(\kka / \ka\right)$ in the Weyl group $W(E_6)$. Let $\Gamma_1$ and $\Gamma_2$ be the Galois groups of equations \eqref{EqEckardt} and \eqref{Eqtangent} respectively. Then in the notation of Section $4$ one has the following:

\begin{itemize}
\item if $\pi_1(\Gamma)$ is trivial then $\Gamma_1$ and $\Gamma_2$ are trivial;
\item if $\pi_1(\Gamma) = \langle cs \rangle \cong \CG_2$ then $\Gamma_1$ is trivial and $\Gamma_2 \cong \CG_2$;
\item if $\pi_1(\Gamma) = \langle a^2b \rangle \cong \CG_3$ then $\Gamma_1$ is trivial and $\Gamma_2 \cong \CG_3$;
\item if $\pi_1(\Gamma) = \langle ab \rangle \cong \CG_3$ then $\Gamma_1 \cong \CG_3$ and $\Gamma_2$ is trivial;
\item if $\pi_1(\Gamma) = \langle a \rangle \cong \CG_3$ then $\Gamma_1 \cong \CG_3$, $\Gamma_2 \cong \CG_3$ and equations \eqref{EqEckardt} and \eqref{Eqtangent} have the same splitting field;
\item if $\pi_1(\Gamma) = \langle a^2b, cs \rangle \cong \SG_3$ then $\Gamma_1$ is trivial and $\Gamma_2 \cong \SG_3$;
\item if $\pi_1(\Gamma) = \langle ab, cs \rangle \cong \CG_6$ then $\Gamma_1 \cong \CG_3$ and $\Gamma_2 \cong \CG_2$;
\item if $\pi_1(\Gamma) = \langle a, b \rangle \cong \CG_3^2$ then $\Gamma_1 \cong \CG_3$, $\Gamma_2 \cong \CG_3$ and equations \eqref{EqEckardt} and \eqref{Eqtangent} have different splitting fields;
\item if $\pi_1(\Gamma) = \langle a, b, cs \rangle \cong \CG_3^2 \rtimes \CG_2$ then $\Gamma_1 \cong \CG_3$ and $\Gamma_2 \cong \SG_3$.

\end{itemize}

\end{lemma}

\begin{proof}
Note that the group $\Gamma_1$ permutes the Eckardt points $e_1$, $e_2$ and $e_3$, and the group $\Gamma_2$ permutes three $(-1)$-curves passing through an Eckardt point. In the notation of Remark~\ref{ExpNot} one can see that the elements $a^2b$ and $cs$ of $W(E_6)$ preserve the Eckardt points $e_1$, $e_2$ and $e_3$, and permute the $(-1)$-curves passing through each of them. Thus the availability of elements conjugate to $a^2b$ and $cs$ in $\pi_1(\Gamma)$ is equivalent to the availability of elements of order $3$ and $2$ in $\Gamma_2$ respectively.

The element $ab$ permutes three $(-1)$-curves $E_{14}$, $E_{25}$ and $E_{36}$. These curves lie in a plane given by $\mu_i x = \lambda_i y$, where $(\lambda_i : \mu_i)$ is a root of equation \eqref{Eqtangent}. Similarly, the element $ab$ preserves the other roots of equation \eqref{Eqtangent}. Thus the availability of the element $ab$ in $\pi_1(\Gamma)$ is equivalent to the availability of an element of order $3$ in $\Gamma_1$.

The group $\langle a, b, cs \rangle \cong \CG_3^2 \rtimes \CG_2$ is generated by the elements $ab$, $a^2b$ and $cs$. So for any subgroup of $\langle a, b, cs \rangle$ one can obtain the result of this lemma.

\end{proof}

Now we want to find geometric interpretation of the actions of the elements $r$ and $s$.

Consider the class $L$ in $\Pic(X)$. We have
$$
rL = 4L - \sum \limits_{i=1}^3 E_i - 2\sum \limits_{i=4}^6 E_i, \qquad r^2L = 4L - 2\sum \limits_{i=1}^3 E_i - \sum \limits_{i=4}^6 E_i,
$$
$$
sL = 5L - 2\sum \limits_{i=1}^6 E_i, \qquad srL = 2L - \sum \limits_{i=4}^6 E_i, \qquad sr^2L = 2L - \sum \limits_{i=1}^3 E_i.
$$

The three fixed points of $G$ on $\XX$ lie on the line $z = t = 0$. We denote these points by $q_1$, $q_2$ and $q_3$ given by the equation
\begin{equation}
\label{Eqfixed}
P(x : y) = 0.
\end{equation}
\noindent There are two $G$-invariant hyperplane sections $z = 0$ and $t = 0$ passing through the fixed points of $G$. We denote these sections by $C_1$ and $C_2$.

Let $h: \XX \rightarrow \Pro^2_{\kka}$ be a $G$-equivariant blowup of $\Pro^2_{\kka}$ at six points $p_1$, $p_2$, $p_3$, $p_4$, $p_5$ and $p_6$ and $l$ be a class of line on $\Pro^2_{\kka}$. Then $G$ has three fixed points on $\Pro^2_{\kka}$. For each two of these fixed points there is exactly one $G$-invariant curve passing through these two points that belongs to one of the following six classes: a line, a quadric passing through $p_1$, $p_2$ and~$p_3$, a quadric passing through $p_4$, $p_5$ and $p_6$, a quartic passing through $p_4$, $p_5$ and $p_6$ and having nodes at $p_1$, $p_2$ and $p_3$, a quartic passing through $p_1$, $p_2$ and $p_3$ and having nodes at $p_4$, $p_5$ and $p_6$ or a quintic having nodes at $p_1$, $p_2$, $p_3$, $p_4$, $p_5$ and $p_6$. Proper transforms of these curves on $X$ are $G$-invariant and can be permuted by the group $\Gamma$. Denote these curves by $R_{ij}^K$, where $K$ is a class of curve in $\Pic(\XX)$ and $i$ and $j$ are indices of points $q_i$ and $q_j$, which $R_{ij}^K$ is passing through.

\begin{lemma}
\label{18invariant}
We can choose notation in such way that the following conditions hold:
\begin{itemize}
\item the curve $C_1$ is tangent to the curves $R_{12}^L$, $R_{12}^{rL}$, $R_{12}^{r^2L}$, $R_{13}^{sL}$, $R_{13}^{srL}$, $R_{13}^{sr^2L}$ at the point $q_1$, tangent to the curves $R_{23}^L$, $R_{23}^{rL}$, $R_{23}^{r^2L}$, $R_{12}^{sL}$, $R_{12}^{srL}$, $R_{12}^{sr^2L}$ at the point $q_2$ and tangent to the curves $R_{13}^L$, $R_{13}^{rL}$, $R_{13}^{r^2L}$, $R_{23}^{sL}$, $R_{23}^{srL}$, $R_{23}^{sr^2L}$ at the point $q_3$;
\item the curve $C_2$ is tangent to the curves $R_{13}^L$, $R_{13}^{rL}$, $R_{13}^{r^2L}$, $R_{12}^{sL}$, $R_{12}^{srL}$, $R_{12}^{sr^2L}$ at the point $q_1$, tangent to the curves $R_{12}^L$, $R_{12}^{rL}$, $R_{12}^{r^2L}$, $R_{23}^{sL}$, $R_{23}^{srL}$, $R_{23}^{sr^2L}$ at the point $q_2$ and tangent to the curves $R_{23}^L$, $R_{23}^{rL}$, $R_{23}^{r^2L}$, $R_{13}^{sL}$, $R_{13}^{srL}$, $R_{13}^{sr^2L}$ at the point $q_3$.
\end{itemize}
\end{lemma}

\begin{proof}
One has $C_1 \cdot R_{12}^L = C_2 \cdot R_{12}^L = 3$. Note that the curves $C_1$, $C_2$ and $R_{12}^K$ pass through the $G$-fixed points $q_1$ and $q_2$, therefore $R_{12}^L$ can not meet $C_1$ and $C_2$ at any other point since that point should be $G$-invariant. Therefore $R_{12}^L$ is tangent to $C_1$ and $C_2$. The curves $C_1$ and $C_2$ have different tangents at the points $q_i$. Thus we can assume that $R_{12}^L$ is tangent to $C_1$ at $q_1$ and tangent to $C_2$ at $q_2$.

In the same way we can show that for any class
$$
K \in \{L, rL, r^2L, sL, srL, sr^2L\}
$$
\noindent the curve $R_{ij}^{K}$ is tangent to $C_1$ and $C_2$ at points $q_i$ and $q_j$. One has $R_{ij}^{K} \cdot R_{ij}^{csrK} = 2$ therefore these curves meet each other transversally and have different tangents at points $q_i$ and $q_j$. Moreover, one has $R_{ij}^K \cdot R_{jk}^K = 1$ therefore these curves meet each other transversally and have different tangents at point $q_j$. Lemma \ref{18invariant} follows from these two facts.

\end{proof}

Now we give explicit geometric interpretation of the action of the group $\pi_2(\Gamma)$.

\begin{lemma}
\label{Scase}
In the notation of Section $4$ the group $\pi_2(\Gamma)$ contains an element conjugate to $s$ if and only if the Galois group $\Gamma_3$ of equation \eqref{Eqfixed} is of even order.
\end{lemma}

\begin{proof}
Let the group $\Gamma_3$ contain an element $h$ such that $h(q_2) = q_3$ and $h(q_3) = q_2$. By Lemma \ref{18invariant} the curve $R_{12}^L$ is tangent to $C_2$ at $q_2$. Thus the curve $h\left(R_{12}^L\right)$ is tangent to $h\left(C_2\right) = C_2$ at $q_3$ and passes through $q_1$. Therefore by Lemma \ref{18invariant} the curve $h\left(R_{12}^L\right)$ is $R_{13}^{sL}$, $R_{13}^{srL}$ or $R_{13}^{sr^2L}$. Hence the group $\pi_2(\Gamma)$ contains an element conjugate to $s$.

Now assume that the group $\pi_2(\Gamma)$ contains an element conjugate to $s$. If the Galois group $\Gamma_3$ is of odd order then this element fixes the points $q_1$, $q_2$ and $q_3$. Therefore the curve $R_{12}^L$ is mapped by $s$ to $R_{12}^{sL}$. But $R_{12}^L$ is tangent to $C_1$ at $q_1$ and $R_{12}^{sL}$ is tangent to $C_2$ at $q_2$. This contradiction finishes the proof.

\end{proof}

\begin{lemma}
\label{DP3nonrat}
Let $X$ be a $G$-minimal cubic surface given by equation \eqref{Eqcubic} and the Galois group $\Gamma_3$ of equation \eqref{Eqfixed} is isomorphic to $\CG_2$. Then the quotient $X / G$ is birationally equivalent to a minimal del Pezzo surface $Z$ of degree $4$. In particularly $X / G$ is not $\ka$-rational.
\end{lemma}

\begin{proof}
The Galois group $\Gamma_3$ of equation \eqref{Eqfixed} is isomorphic to $\CG_2$. Therefore we can assume that the $G$-fixed point $q_1$ is defined over $\ka$ and two other $G$-fixed points $q_2$ and $q_3$ are permuted by $\Gamma_3$.

Let $f: X \rightarrow X / G$ be the quotient morphism and
$$
\pi: \widetilde{X / G} \rightarrow X / G
$$
\noindent be the minimal resolution of singularities.

There are three singular points of type $A_2$ on $X / G$, namely $f(q_1)$, $f(q_2)$ and $f(q_3)$. The curves $C_1$, $C_2$ and the point $q_1$ are defined over $\ka$. Thus the irreducible components of $\pi^{-1} f\left( q_1 \right)$ are defined over $\ka$. The group $\Gamma_3 \cong \CG_2$ maps the irreducible components of~$\pi^{-1} f\left( q_2 \right)$ to the irreducible components of~$\pi^{-1} f\left( q_3 \right)$. Therefore, one has
$$
\rho(\widetilde{X / G}) = \rho(X / G) + 4 = \rho(X)^G + 4 = 5.
$$
\noindent As in the proof of Lemma \ref{DP3C35} two curves $\pi^{-1}_* f (C_i)$ are $(-1)$-curves defined over $\ka$. We can contract this pair and get a del Pezzo surface $Y$ such that $K_Y^2 = 3$ and $\rho(Y) = 3$.

The Galois group $\Gamma_3$ acts on the set of $27$ $(-1)$-curves on $Y$. One cannot contract more than four $(-1)$-curves on $Y$ since $\rho(Y) = 3$. But in Table $1$ in \cite[Chapter~IV, $\S5$]{Man74} there is only one class of elements of order $2$ satisfying this property. This class corresponds to the $11$-th row of the table. For this class one cannot contract more than one $(-1)$-curve on $Y$ (see the second column of the table). Therefore one can contract this curve on $Y$ and get a minimal del Pezzo surface $Z$ of degree $4$ with $\rho(Z) = 2$ admitting a structure of conic bundle. The surface $Z \approx X / G$ is not $\ka$-rational by Theorem \ref{ratcrit}.

\end{proof}

Assume that the $G$-fixed point $q_1$ on the cubic surface $X$ is defined over $\ka$. Then after the change of coordinates this cubic is given by the equation
$$
wx(x^2 - \lambda y^2) + zt(ux + vy) + z^3 + \alpha t^3 = 0.
$$

For this cubic surface the following lemma holds.

\begin{lemma}
\label{Rcase}
In the notation of Section $4$ the group $\pi_2(\Gamma)$ contains $r$ if and only if the Galois group $\Gamma_4$ of the equation
\begin{equation}
\label{Eqfamily}
4 \alpha \mu^3 - \left(u^2 - \frac{v^2}{\lambda}\right) \mu^2 - 2 uw \mu - w^2 = 0.
\end{equation}
\noindent contains an element of order $3$.
\end{lemma}

\begin{proof}
Note that the divisors $R_{23}^L + R_{23}^{sL}$, $R_{23}^{rL} + R_{23}^{sr^2L}$ and $R_{23}^{r^2L} + R_{23}^{srL}$ are linearly equivalent to $-2K_X$. Therefore these $G$-invariant pairs of curves passing through $q_2$ and $q_3$ are cut from $X$ by the quadric surfaces of the following form
$$
\mu(x^2 - \lambda y^2) = zt.
$$

Let us find reducible members in these family of curves. One has
$$
wx(x^2 - \lambda y^2)t^3 + \mu (ux + vy)(x^2 - \lambda y^2)t^3 + \mu^3(x^2 - \lambda y^2)^3 + \alpha t^6 = 0.
$$
\noindent If the polynomial in the left hand side fo the latter equation is reducible over $\ka(x, y, t)$ then it factorizes in the following way
$$
\left(A(x - y\sqrt{\lambda})(x^2 - \lambda y^2) + \sqrt{\alpha} t^3\right)\left(B(x + y\sqrt{\lambda})(x^2 - \lambda y^2) + \sqrt{\alpha} t^3\right) = 0
$$
\noindent and therefore we have $AB = \mu^3$ and
$$
A(x - y\sqrt{\lambda})\sqrt{\alpha} + B(x + y\sqrt{\lambda})\sqrt{\alpha} = \mu (ux + vy) + wx.
$$
Therefore the following system of equations holds
$$
\begin{cases}
(A + B)\sqrt{\alpha} = \mu u + w, \\
(B - A)\sqrt{\lambda \alpha} = \mu v.
\end{cases}
$$
Solving this system one has
$$
A = \frac{(\mu u + w)\sqrt{\lambda} - \mu v}{2\sqrt{\lambda \alpha}}, \qquad B = \frac{(\mu u + w)\sqrt{\lambda} + \mu v}{2\sqrt{\lambda \alpha}}.
$$
Since $AB = \mu^3$, the reducible members of the linear system $|-2K_X|$ passing through $q_2$ and $q_3$ are given by equation \eqref{Eqfamily}.

The roots of this equation correspond to the pairs of curves $R_{23}^L$ and $R_{23}^{sL}$, $R_{23}^{rL}$ and $R_{23}^{sr^2L}$, $R_{23}^{r^2L}$ and $R_{23}^{srL}$ which are cyclically permuted by $\Gamma$ if and only if the group $\pi_2(\Gamma)$ contains $r$.
\end{proof}

\begin{remark}
At the beginning of this section we assume that the field $\ka$ contains $\omega$. For any field $\ka$ the action of a generator of $G$ can be written as
$$
(x : y : z : t) \mapsto (y : z : x : t).
$$
\noindent One can remake the computations (which is much more complicated) for this action. Then Lemmas \ref{X3Eckardt}, \ref{P1case}, \ref{18invariant} and \ref{Rcase} hold. But Lemma \ref{Scase} does not hold since the curves $C_1$ and $C_2$ are not defined over $\ka$.
\end{remark}

\section{Examples}

In this section we construct explicit examples of quotients of del Pezzo surfaces of degree~$3$ by a group $G \cong \CG_3$ acting as in type $5$ of Table~\ref{table2}. We use the notation of Section $5$.

\begin{lemma}
\label{X3ratrat}
Let $X$ be a cubic surface given by equation \eqref{Eqcubic}. Suppose that the Galois groups $\Gamma_1$, $\Gamma_2$, $\Gamma_3$ of equations \eqref{EqEckardt}, \eqref{Eqtangent}, \eqref{Eqfixed} are trivial and the Galois group $\Gamma_4$ of equation \eqref{Eqfamily} contains an element of order $3$. Then the surface $X$ is $G$-minimal and $\ka$-rational, and the quotient $X / G$ is also $\ka$-rational.
\end{lemma}

\begin{proof}
The group $\Gamma_1$ is trivial. Therefore $X(\ka)$ contains the points $e_1$, $e_2$ and $e_3$.

By Lemmas \ref{P1case}, \ref{Scase} and \ref{Rcase} the group $\Gamma$ is conjugate to $\langle r \rangle$. Therefore one can Galois equivariantly contract the curves $E_1$, $L_{23}$ and $Q_1$ and get a del Pezzo surface of degree $6$ which is $\ka$-rational by Theorem \ref{ratcrit}.

The image of the group $G$ in the Weyl group $W(E_6)$ is $\langle ab \rangle$ thus $X$ is $G$-minimal by Corollary \ref{Galmin4}.

Let $f: X \rightarrow X / G$ be the quotient morphism and
$$
\pi: \widetilde{X / G} \rightarrow X / G
$$
\noindent be the minimal resolution of singularities. The group $\Gamma_3$ is trivial. Therefore the points $q_1$, $q_2$ and $q_3$ are defined over $\ka$. Thus $\rho(\widetilde{X / G}) = 7$, and $\widetilde{X / G}$ and $X / G$ are $\ka$-rational by Corollary \ref{piccrit}.

\end{proof}

\begin{example}
\label{exratrat}
If the field $\ka$ contains $\omega$ and an element $\nu$ such that $\sqrt[3]{\nu} \notin \ka$ then the cubic surface given by the equation
$$
2\nu x(x^2-y^2) +z^3 + t^3 = 0
$$
\noindent satisfies the conditions of Lemma \ref{X3ratrat}.
\end{example}

\begin{lemma}
\label{X3ratnrat}
Let $X$ be a cubic surface given by equation \eqref{Eqcubic}. Suppose that the Galois groups $\Gamma_1$, $\Gamma_2$ of equations \eqref{EqEckardt}, \eqref{Eqtangent} are trivial, the Galois group $\Gamma_3$ of equation \eqref{Eqfixed} is isomorphic to $\CG_2$ and the Galois group $\Gamma_4$ of equation \eqref{Eqfamily} contains an element of order $3$. Then the surface $X$ is $G$-minimal and $\ka$-rational, and the quotient $X / G$ is not $\ka$-rational.
\end{lemma}

\begin{proof}
The group $\Gamma_1$ is trivial. Therefore $X(\ka)$ contains the points $e_1$, $e_2$ and $e_3$.

By Lemmas \ref{P1case}, \ref{Scase} and \ref{Rcase} the group $\Gamma$ is conjugate to $\langle r, s \rangle$. Therefore one can Galois equivariantly contract the curves $E_1$, $L_{23}$ and $Q_1$ and get a del Pezzo surface of degree $6$ which is $\ka$-rational by Theorem \ref{ratcrit}.

The image of the group $G$ in the Weyl group $W(E_6)$ is $\langle ab \rangle$ thus $X$ is $G$-minimal by Corollary \ref{Galmin4}.

The quotient $X / G$ is not $\ka$-rational by Lemma \ref{DP3nonrat}.

\end{proof}

\begin{example}
\label{exratnrat}
Suppose that the field $\ka$ contains $\omega$, and does not contain $\sqrt{2}$ and any root of the equation
$$
4\mu^3 - 9\mu^2 -6\mu - 1 = 0.
$$
\noindent Then the cubic surface given by the equation
$$
x(x^2 - 2y^2) + 3xzt + z^3 + t^3 = 0
$$
\noindent satisfies the conditions of Lemma \ref{X3ratnrat}.
\end{example}

\begin{lemma}
\label{X3nratrat}
Let $X$ be a cubic surface given by equation \eqref{Eqcubic}. Suppose that the Galois groups $\Gamma_1$ and $\Gamma_3$ of equations \eqref{EqEckardt} and \eqref{Eqfixed} are trivial, the Galois group $\Gamma_2$ of equation~\eqref{Eqtangent} is isomorphic to $\CG_2$ and the Galois group $\Gamma_4$ of equation \eqref{Eqfamily} contains an element of order $3$. Then the surface $X$ is $G$-minimal and not $\ka$-rational, and the quotient $X / G$ is $\ka$-rational.
\end{lemma}

\begin{proof}
The group $\Gamma_1$ is trivial. Therefore $X(\ka)$ contains the points $e_1$, $e_2$ and $e_3$.

By Lemmas \ref{P1case}, \ref{Scase} and \ref{Rcase} the group $\Gamma$ is conjugate to $\langle c, r \rangle$. Therefore $X$ is not $\ka$-rational by Lemma \ref{Galnonrat2}.

The image of the group $G$ in the Weyl group $W(E_6)$ is $\langle ab \rangle$ thus $X$ is $G$-minimal by Corollary \ref{Galmin4}.

Let $f: X \rightarrow X / G$ be the quotient morphism and
$$
\pi: \widetilde{X / G} \rightarrow X / G
$$
\noindent be the minimal resolution of singularities. The group $\Gamma_3$ is trivial. Therefore the points $q_1$, $q_2$ and $q_3$ are defined over $\ka$. Thus $\rho(\widetilde{X / G}) = 7$, and $\widetilde{X / G}$ and $X / G$ are $\ka$-rational by Corollary \ref{piccrit}.

\end{proof}

\begin{example}
\label{exnratrat}
In the assumptions of Example \ref{exratnrat} the cubic surface given by the equation
$$
x(x^2 - y^2) + 3xzt + z^3 + t^3 = 0
$$
\noindent satisfies the conditions of Lemma \ref{X3nratrat}.
\end{example}

\begin{lemma}
\label{X3nratnrat}
Let $X$ be a cubic surface given by equation \eqref{Eqcubic}. Suppose that the Galois group $\Gamma_1$ of equation \eqref{EqEckardt} is trivial, the Galois group $\Gamma_2$ of equation \eqref{Eqtangent} is isomorphic to $\CG_2$ and the Galois group $\Gamma_3$ of equation \eqref{Eqfixed} is isomorphic to $\CG_2$. Then the surface $X$ is $G$-minimal and not $\ka$-rational, and the quotient $X / G$ is also not $\ka$-rational.
\end{lemma}

\begin{proof}
The group $\Gamma_1$ is trivial. Therefore $X(\ka)$ contains the points $e_1$, $e_2$ and $e_3$.

By Lemmas \ref{P1case} and \ref{Scase} the group $\Gamma$ contains a subgroup $\langle cs \rangle \cong \CG_2$. Therefore $X$ is not $\ka$-rational by Lemma \ref{Galnonrat2}.

The image of the group $G$ in the Weyl group $W(E_6)$ is $\langle ab \rangle$ thus $X$ is $G$-minimal by Corollary \ref{Galmin2}.

The quotient $X / G$ is not $\ka$-rational by Lemma \ref{DP3nonrat}.

\end{proof}

\begin{example}
\label{exnratnrat}
If the field $\ka$ contains $\omega$ and an element $\lambda$ such that $\sqrt{\lambda} \notin \ka$ then the cubic surface given by the equation
$$
ux(x^2- \lambda y^2) +z^3 + t^3 = 0
$$
\noindent satisfies the conditions of Lemma \ref{X3nratnrat}.
\end{example}

\begin{remark}
Note that the conditions of Examples \ref{exratrat}, \ref{exratnrat}, \ref{exnratrat}, \ref{exnratnrat} hold for $\ka = \Q(\omega)$.
\end{remark}

\bibliographystyle{alpha}
\bibliography{my_ref}
\end{document}